\documentclass[11pt, reqno]{amsart}
\usepackage[dvipsnames]{xcolor}
\usepackage{graphicx}
\usepackage{caption}
\usepackage{array}
\usepackage{amssymb}
\usepackage{amsthm}
\usepackage{amsmath}
\usepackage{verbatim}
\usepackage{color,soul}
\usepackage{hhline}
\usepackage{tikz}
\usepackage[shortlabels]{enumitem}
\usepackage{diagbox}
\usetikzlibrary{quotes,angles,positioning}
\usepackage{lscape}
\usetikzlibrary{arrows.meta}
\usepackage[normalem]{ulem}
\usepackage{mathtools, nccmath}
\usepackage[shortlabels]{enumitem}
\usepackage[bordercolor=gray!20,backgroundcolor=blue!10,linecolor=blue!50,textsize=footnotesize,textwidth=0.8in]{todonotes}

\newtheorem{thm}{Theorem}[section]

\newtheorem{corollary}[thm]{Corollary}

\newtheorem{lemma}[thm]{Lemma}

\newtheorem{proposition}[thm]{Proposition}
\newtheorem{claim}[thm]{Claim}
\newtheorem{theorem}[thm]{Theorem}

\theoremstyle{definition}

\newtheorem{definition}[thm]{Definition}

\newtheorem{example}[thm]{Example}

\newtheorem{remark}[thm]{Remark}

\numberwithin{equation}{section}
\newenvironment{acknowledgements}{%
  \begin{abstract}
}{%
  \end{abstract}
}

\newcommand{\Cf}{{\tt CnFct}(B_3)}
\newcommand{\CF}{{\tt CnFct}(B_4)}
\newcommand{\CFn}{{\tt CnFct}(B_n)}
\newcommand{\A}{{\mathcal A}}
\newcommand{\X}{{\mathcal X}}
\newcommand{\Y}{{\mathcal Y}}
\newcommand{\Z}{{\mathcal Z}}
\newcommand{\B}{{\mathcal B}}
\newcommand{\C}{{\mathcal C}}
\newcommand{\LCF}{{\tt LCF}}
\newcommand{\BG}{{\tt BGen}}
\newcommand{\red}{{\tt red}}
\newcommand{\Red}{{\tt Red}}
\newcommand{\SSS}{{\tt SSS}}
\newcommand{\Wds}{{\tt Wds}}
\newcommand{\SL}{{\tt SL}}
\newcommand{\D}{{\mathcal D}}
\newcommand{\nb}{{\tt nb}}
\newcommand{\ch}{{\tt cvh}}
\newcommand{\Int}{{\text{int}}}

\providecommand\st{}
\newcommand\RelSymbol[1][]{%
\nonscript\:#1\vert
\allowbreak
\nonscript\:
\mathopen{}}
\DeclarePairedDelimiterX\GenRels[1]\langle\rangle{%
\renewcommand\st{\RelSymbol[\delimsize]}
#1}

\newcommand{\tik}[1]{%
  \begin{tikzpicture}[baseline=-\dimexpr\fontdimen22\textfont2\relax]
  #1
  \end{tikzpicture}%
}

\newcommand{\aone}{%
  \tik{
    \coordinate (1) at (0,-.15);
\coordinate (2) at (.25,-.15);
\coordinate (3) at (.25,.1);
\coordinate (4) at (0,.1);
\foreach \i/\Position in {1/below, 2/below, 3/above, 4/above}{
    \fill (\i) circle (1pt) node [\Position] {};}
    \draw (1)--(2)--cycle;}%
}
\newcommand{\aonehigh}{%
  \tik{
    \coordinate (1) at (0,-.15);
\coordinate (2) at (.25,-.15);
\coordinate (3) at (.25,.1);
\coordinate (4) at (0,.1);
\foreach \i/\Position in {1/below, 2/below, 3/above, 4/above}{
    \fill (\i) circle (1pt) node [\Position] {};}
    \draw (1)--(2)--cycle;%
    \draw[yellow, ultra thick] (1)--(3)--(4)--cycle;}%
}
\newcommand{\atwo}{%
  \tik{
    \coordinate (1) at (0,-.15);
\coordinate (2) at (.25,-.15);
\coordinate (3) at (.25,.1);
\coordinate (4) at (0,.1);
\foreach \i/\Position in {1/below, 2/below, 3/above, 4/above}{
    \fill (\i) circle (1pt) node [\Position] {};}
    \draw (2)--(3)--cycle;}%
}
\newcommand{\athree}{%
  \tik{
    \coordinate (1) at (0,-.15);
\coordinate (2) at (.25,-.15);
\coordinate (3) at (.25,.1);
\coordinate (4) at (0,.1);
\foreach \i/\Position in {1/below, 2/below, 3/above, 4/above}{
    \fill (\i) circle (1pt) node [\Position] {};}
    \draw (3)--(4)--cycle;}%
}

\newcommand{\afour}{%
  \tik{
    \coordinate (1) at (0,-.15);
\coordinate (2) at (.25,-.15);
\coordinate (3) at (.25,.1);
\coordinate (4) at (0,.1);
\foreach \i/\Position in {1/below, 2/below, 3/above, 4/above}{
    \fill (\i) circle (1pt) node [\Position] {};}
    \draw (4)--(1)--cycle;}%
}
\newcommand{\afourhigh}{%
  \tik{
    \coordinate (1) at (0,-.15);
\coordinate (2) at (.25,-.15);
\coordinate (3) at (.25,.1);
\coordinate (4) at (0,.1);
\foreach \i/\Position in {1/below, 2/below, 3/above, 4/above}{
    \fill (\i) circle (1pt) node [\Position] {};}
    \draw (4)--(1)--cycle;
    \draw[yellow, ultra thick](4)--(3)--(2)--cycle;}%
}
\newcommand{\e}{%
  \tik{
    \coordinate (1) at (0,-.15);
\coordinate (2) at (.25,-.15);
\coordinate (3) at (.25,.1);
\coordinate (4) at (0,.1);
\foreach \i/\Position in {1/below, 2/below, 3/above, 4/above}{
    \fill (\i) circle (1pt) node [\Position] {};}}%
}
\newcommand{\bone}{%
  \tik{
    \coordinate (1) at (0,-.15);
\coordinate (2) at (.25,-.15);
\coordinate (3) at (.25,.1);
\coordinate (4) at (0,.1);
\foreach \i/\Position in {1/below, 2/below, 3/above, 4/above}{
    \fill (\i) circle (1pt) node [\Position] {};}
    \draw (1)--(3)--cycle;}%
}
\newcommand{\bonehigh}{%
  \tik{
    \coordinate (1) at (0,-.15);
\coordinate (2) at (.25,-.15);
\coordinate (3) at (.25,.1);
\coordinate (4) at (0,.1);
\foreach \i/\Position in {1/below, 2/below, 3/above, 4/above}{
    \fill (\i) circle (1pt) node [\Position] {};}
    \draw (1)--(3)--cycle;
    \draw[yellow, ultra thick] (1)--(4);%
    \draw[yellow, ultra thick] (2)--(3);}%
}
\newcommand{\btwo}{%
  \tik{
    \coordinate (1) at (0,-.15);
\coordinate (2) at (.25,-.15);
\coordinate (3) at (.25,.1);
\coordinate (4) at (0,.1);
\foreach \i/\Position in {1/below, 2/below, 3/above, 4/above}{
    \fill (\i) circle (1pt) node [\Position] {};}
    \draw (2)--(4)--cycle;}%
}
\newcommand{\atwoaone}{%
  \tik{
    \coordinate (1) at (0,-.15);
\coordinate (2) at (.25,-.15);
\coordinate (3) at (.25,.1);
\coordinate (4) at (0,.1);
\foreach \i/\Position in {1/below, 2/below, 3/above, 4/above}{
    \fill (\i) circle (1pt) node [\Position] {};}
    \draw (1)--(2)--(3)--(1)--cycle;}%
}
\newcommand{\atwoaonehigh}{%
  \tik{
    \coordinate (1) at (0,-.15);
\coordinate (2) at (.25,-.15);
\coordinate (3) at (.25,.1);
\coordinate (4) at (0,.1);
\foreach \i/\Position in {1/below, 2/below, 3/above, 4/above}{
    \fill (\i) circle (1pt) node [\Position] {};}
    \draw (1)--(2)--(3)--(1)--cycle;
    \draw[yellow, ultra thick](1)--(4);}%
}
\newcommand{\athreeatwo}{%
  \tik{
    \coordinate (1) at (0,-.15);
\coordinate (2) at (.25,-.15);
\coordinate (3) at (.25,.1);
\coordinate (4) at (0,.1);
\foreach \i/\Position in {1/below, 2/below, 3/above, 4/above}{
    \fill (\i) circle (1pt) node [\Position] {};}
    \draw (2)--(3)--(4)--(2)--cycle;}%
}
\newcommand{\afourathree}{%
  \tik{
    \coordinate (1) at (0,-.15);
\coordinate (2) at (.25,-.15);
\coordinate (3) at (.25,.1);
\coordinate (4) at (0,.1);
\foreach \i/\Position in {1/below, 2/below, 3/above, 4/above}{
    \fill (\i) circle (1pt) node [\Position] {};}
    \draw (1)--(3)--(4)--(1)--cycle;}%
}
\newcommand{\afourathreehigh}{%
  \tik{
    \coordinate (1) at (0,-.15);
\coordinate (2) at (.25,-.15);
\coordinate (3) at (.25,.1);
\coordinate (4) at (0,.1);
\foreach \i/\Position in {1/below, 2/below, 3/above, 4/above}{
    \fill (\i) circle (1pt) node [\Position] {};}
    \draw (1)--(3)--(4)--(1)--cycle;
    \draw[yellow, ultra thick](2)--(3);}%
}
\newcommand{\aoneafour}{%
  \tik{
    \coordinate (1) at (0,-.15);
\coordinate (2) at (.25,-.15);
\coordinate (3) at (.25,.1);
\coordinate (4) at (0,.1);
\foreach \i/\Position in {1/below, 2/below, 3/above, 4/above}{
    \fill (\i) circle (1pt) node [\Position] {};}
    \draw (1)--(2)--(4)--(1)--cycle;}%
}
\newcommand{\aoneafourhigh}{%
  \tik{
    \coordinate (1) at (0,-.15);
\coordinate (2) at (.25,-.15);
\coordinate (3) at (.25,.1);
\coordinate (4) at (0,.1);
\foreach \i/\Position in {1/below, 2/below, 3/above, 4/above}{
    \fill (\i) circle (1pt) node [\Position] {};}
    \draw (1)--(2)--(4)--(1)--cycle;
    \draw[yellow, ultra thick](4)--(3);}%
}
\newcommand{\aoneathree}{%
  \tik{
    \coordinate (1) at (0,-.15);
\coordinate (2) at (.25,-.15);
\coordinate (3) at (.25,.1);
\coordinate (4) at (0,.1);
\foreach \i/\Position in {1/below, 2/below, 3/above, 4/above}{
    \fill (\i) circle (1pt) node [\Position] {};}
    \draw (1)--(2)--cycle;
    \draw (3)--(4)--cycle;}%
}
\newcommand{\aoneathreehigh}{%
  \tik{
    \coordinate (1) at (0,-.15);
\coordinate (2) at (.25,-.15);
\coordinate (3) at (.25,.1);
\coordinate (4) at (0,.1);
\foreach \i/\Position in {1/below, 2/below, 3/above, 4/above}{
    \fill (\i) circle (1pt) node [\Position] {};}
    \draw (1)--(2)--cycle;
    \draw (3)--(4)--cycle;
    \draw[yellow, ultra thick](1)--(3);}%
}
\newcommand{\atwoafour}{%
  \tik{
    \coordinate (1) at (0,-.15);
\coordinate (2) at (.25,-.15);
\coordinate (3) at (.25,.1);
\coordinate (4) at (0,.1);
\foreach \i/\Position in {1/below, 2/below, 3/above, 4/above}{
    \fill (\i) circle (1pt) node [\Position] {};}
    \draw (2)--(3)--cycle;
    \draw (1)--(4)--cycle;}%
}
\newcommand{\delt}{%
  \tik{
    \coordinate (1) at (0,-.15);
\coordinate (2) at (.25,-.15);
\coordinate (3) at (.25,.1);
\coordinate (4) at (0,.1);
\foreach \i/\Position in {1/below, 2/below, 3/above, 4/above}{
    \fill (\i) circle (1pt) node [\Position] {};}
    \draw (1)--(2)--(3)--(4)--cycle;}%
}

\begin{document}
\title{Birman-Ko-Lee left canonical form and its applications}
\author{Michele Capovilla-Searle}
\author{Keiko Kawamuro}
\author{Rebecca Sorsen}
\address {Department of Mathematics, University of Iowa, Iowa City, IA 52242}
\email{michele-capovilla-searle@uiowa.edu}
\email{keiko-kawamuro@uiowa.edu}
\email{rebecca-sorsen@uiowa.edu}
\date{}
\maketitle
\newcolumntype{P}[1]{>{\centering\arraybackslash}p{#1}}

\begin{abstract}
Using Birman, Ko, and Lee's left canonical form of a braid, we 
characterize almost strongly quasipositive braids and 
give estimates of the fractional Dehn twist coefficient.  
\end{abstract}
\tableofcontents

\section{Introduction}
The braid groups $B_n$ for $n=1,2,3,\dots$ were originally introduced by Emil Artin in \cite{Artin}. 
The following presentation is known as the standard presentation of $B_n$ with the Artin generators $\sigma_1,\dots,\sigma_{n-1}$. 
\[B_n= \GenRels*{ \sigma_1,\dots,\sigma_{n-1} \st
\begin{medsize}
\begin{aligned} & \sigma_i\sigma_j=\sigma_j\sigma_i,\enspace |i-j|>1 \\[-0.5ex] %
 & \sigma_i\sigma_{i+1}\sigma_i=\sigma_{i+1}\sigma_i\sigma_{i+1}, \enspace i=1,\dots,n-2%
\end{aligned}
\end{medsize}
} \]

In the standard presentation of $B_n,$ Artin generators only swap adjacent strands.
In the band generator presentation, any two strands are allowed to swap. For each $i,j$ with $1\leq i<j\leq n$, define
\begin{equation*}
    a_{ij}=(\sigma_{j-1}\sigma_{j-2}\cdots\sigma_{i+1})\sigma_i(\sigma_{j-1}\cdots\sigma_{j-2}\sigma_{i+1})^{-1}.
\end{equation*}
Here, $a_{ij}$ is the element in $B_n$ which swaps the $i$th and $j$th strand, while leaving all other strands fixed. In addition, the two strands being interchanged must lie in front of all fixed strands. This gives us the band generator presentation (or Birman-Ko-Lee presentation) \cite{Birman} as follows:
\begin{equation*}
B_n= \GenRels*{ a_{ij} \st
\begin{medsize}
\begin{aligned} & a_{jk}a_{ij}=a_{ij}a_{ik}=a_{ik}a_{jk},\enspace i<j<k \\[-0.5ex] %
 & a_{ij}a_{kl}=a_{kl}a_{ij}, \enspace i<j<k<l%
\end{aligned}
\end{medsize}
} 
\end{equation*}

Word problems and conjugacy problems have been central problems in the study of braids. 
Given two $n$-braid words $w$ and $w'$ determining whether $w=w'$ the same braid element in $B_n$ is called the word problem, and determining whether $w = v w' v^{-1}$ for some $v \in B_n$ i.e., $w$ and $w'$ are conjugate, is called the conjugacy problem. 

These problems have been solved by a number of people, including Artin \cite{Artin}, Garside \cite{Garside}, Elrifai and Morton \cite{ElrifaiMorton}, Xu \cite{Xu}, Kang, Ko, and Lee \cite{KangKoLee}, and Birman, Ko and Lee \cite{Birman}. The latter three papers used band generators to solve these problems.

Birman-Ko-Lee's left canonical form, $\LCF(\beta)$, was used to solve the word problem using band generator techniques.
Namely, Birman, Ko and Lee, generalizing the earlier work by Xu \cite{Xu} and Kang, Ko and Lee \cite{KangKoLee}, proved that for $n$-braid words $w$ and $w'$, we have $w=w'$ in $B_n$ if and only if $\LCF(w)=\LCF(w')$. 
They also solved the conjugacy problem using the left canonical form. See Theorem~\ref{thm:LCF} for the  definition of $\LCF(\beta)$.

Birman, Ko, and Lee's definition of $\LCF(\beta)$ is algebraic. In this paper, we study $\LCF(\beta)$ using non-crossing partition diagrams. These  diagrams can be found in the literature, including the book \cite{OrderingBraids} by Dehornoy, Dynnikov, Rolfsen and Wiest, as well as a paper by Calvez and Wiest \cite{CalvezWiest}.

We will give an outline of the paper and state some of our main results.

In \S\ref{Sec2}, we start with reviewing band generators and non-crossing partition diagrams.  We then define canonical factors that play key role in $\LCF(\beta)$ and relate canonical factors with non-crossing partition diagrams.

In \S\ref{Sec:3}, we introduce a partial ordering $\prec$ on the set of canonical factors of $B_n$, denoted $\CFn$. The definition of $\prec$ is algebraic. We give a graphical interpretation.

\noindent{\bf Theorem \ref{thm:subset}.} {\em 
For canonical factors $A$ and $B$ we have $A\prec B$ if and only if their convex hulls satisfy $\ch(A) \subset \ch(B)$.
}

To this end, we introduce a binary operation $\diamond$ on the set $\CFn$. As a biproduct, we obtain a useful corollary:

\noindent{\bf Corollary~\ref{cor:complementary}.} {\em
For every canonical factor $A\in \CFn$ there exists a $B\in\CFn$ such that $A\diamond B=AB=\delta.$ }

In \S\ref{Sec:OrderedSet}, using the partial ordering $\prec$ in \S~\ref{Sec:3},  we introduce another ordering $\Rightarrow$ on the set $\Wds(\beta)$ of braid words representing the braid $\beta$. 
We give detailed examples how the ordering works in the braid group $B_4$. 

In \S\ref{sec:LCF-def}, we study the left canonical form. 
The set $\Wds(\beta)$ is an infinite set; however, Birman, Ko, and Lee's theorem \cite{Birman} states that with respect to the ordering $\Rightarrow$, there is a unique maximal element in $\Wds(\beta)$. This unique maximal element is called the {\em left canonical form} of $\beta$ and is denoted by $\LCF(\beta)$. 
We also review Kang, Ko and Lee's algorithm \cite{KangKoLee} for the left canonical form and compute examples using non-crossing partition diagrams.

In \S\ref{Sec:6}, we begin to discuss the notion of positivity of knots and links. In this case, positivity refers to the property that all crossings of a link have the same sign. In the literature, multiple notions of positivity for braids have been studied, including positive (P) braids that is the monoid generated by Artin generators $\{\sigma_i\}_{i=1}^{n-1}$, quasipositive (QP) braids that is the monoid normally generated by the Artin generators, and strongly quasipositive (SQP) braids that is the monoid $B_n^+$ generated by the positive band generators $a_{i, j}$. 

It is also interesting to study the notion of almost positivity. For example, if we allow one crossing to be negative, does this change any of the properties of positivity? 
We call such braids almost strongly quasipositive (ASQP). 
In \cite{Keiko}, Hamer, Ito, and Kawamuro discussed properties and relations among the various notions of positivity and almost positivity. 

These positivity notions are related by inclusions: 
$$
\begin{array}{ccccccc}
&& {\rm QP} & \subset & {\rm AQP} & \subset & B_n \\
&& \cup & & \cup & & \\
{\rm P} &\subset & {\rm SQP}=B_n^+ & \subset & {\rm ASQP} & & 
\end{array}
$$

The {\em infimum of a braid} $\beta$, denoted by 
$\inf(\beta)$, is an invariant of the braid and can be directly read from $\LCF(\beta)$ (cf. Definition~\ref{def:inf-sup} and Theorem~\ref{thm:LCF}). 
We discuss the invariant $\inf(\beta)$ for SQP and ASQP braids. 
For SQP braids, we obtain the following if and only if statement. 

\noindent{\bf Theorem~\ref{Thm:SQ}.} {\em 
An $n$-braid $\beta\in B_n$ is strongly quasipositive if and only if $\inf(\beta)\geq 0$. }

For ASQP braids, we obtain an if and only if statement in Theorem~\ref{thm:strictlyasqp} discussed below.

In \S\ref{Sec:7}, we study the {\em negative band number} $\nb(\beta)$.  
It is the minimal number of negative bands for $\beta\in B_n$, i.e.,
$$
\nb(\beta)=\min\left\{ k \, \middle\vert
    \begin{array}{l}
    \text{ $\beta$ is represented by a word in band}\\
    \text{generators containing $k$ negative bands }
    \end{array}
\right\}
$$  
The invariant $\nb(\beta)$ is an upper bound of the defect of the Bennequin inequality \cite{KeikoIto}.
$$ \frac{1}{2}(- \chi(K) - \SL(K)) \leq  \nb[\beta]$$
where $K$ is the knot (or link) type of the braid closure $\hat\beta$, $\chi(K)$ is the maximal Euler characteristic of $K$, and $\SL(K)$ is the maximal self-linking number for $K$.
We give two estimates of the negative band number $\nb(\beta)$ in terms of the left canonical form. The first one is an upper bound of $-\nb(\beta)$. The second one, a lower bound is given in \S\ref{Sec:8}.

\noindent{\bf Theorem~\ref{thm:inequality}.} {\em 
Let $\beta\in B_n$ with $\nb(\beta)\geq 1$. Then $$
0< -\inf(\beta)= |\inf(\beta)| \leq \nb(\beta). 
$$}

In \S~\ref{Sec:8}, we review the reduction operation introduced by Kang, Ko, and Lee \cite{KangKoLee} to solve the shortest word problem for 4-braids. We use the reduction operation to give another estimate of $-\nb(\beta)$ that is complemental to Theorem~\ref{thm:inequality} as follows: 

\noindent{\bf Theorem~\ref{thm:3braid-nb}.} {\em 
Let $\beta$ be an n-braid.  
If $\inf(\beta)<0$ then 
$$\nb(\beta)\leq (n-2) |\inf(\beta)| - \min\{0, \sup(\beta)\}.$$ 
Moreover, the equality holds when $n=3$ and we have
$$\nb(\beta)= |\inf(\beta)| - \min\{0, \sup(\beta)\}.$$} 

In \S\ref{Sec:9}, we discuss more relations between $\inf(\beta)$ and $\nb(\beta)$ for a specific super summit element $\beta\in\SSS(\beta)$.  
We characterize strictly ASQP braids in terms of the left canonical form.

\noindent{\bf Theorem \ref{thm:strictlyasqp}.} {\em 
A braid $\beta\in B_n$ with $n\leq 4$ 
is conjugate to a strictly almost strongly quasipositive braid if and only if every super summit element $\beta'\in \SSS(\beta)$ has $\inf(\beta')=-1$ and $\LCF(\beta')$ contains a canonical factor of word length $n-2$. }

In the above theorem, the restriction on the braid index $n=3,4$ is only required for the only-if ($\Rightarrow$) direction. The statement of the if-direction ($\Leftarrow$) holds for general $n$.

In Section \S\ref{sec:fdtc}, we apply the left canonical form of a braid $\beta$ to study its fractional Dehn twist coefficient (denoted FDTC). Following the description of the fractional Dehn twist coefficient in \cite{HondaKazezMatic} and \cite{ItoKawamuroEssential},
 we first develop tools to easily compute the FDTC of a braid $\beta$. Then, applying the left canonical form, we obtain a bound on the FDTC of $\beta$.

\noindent{\bf Theorem~\ref{FDTC_bound}.} {\em
Suppose  $\LCF(\beta) = \delta^r A_1 \cdots A_k$ then 
$$
\frac{\inf(\beta)}{n} = \frac{r}{n} \leq c(\beta)\leq \frac{r+k}{n} = \frac{\sup(\beta)}{n}.
$$
}

\begin{acknowledgements}
    The authors would like to thank Juan Gonzalez-Meneses for useful conversations during a workshop held at ICERM Brown University in April 2022.
\end{acknowledgements}

\section{Canonical Factors via Diagram}\label{Sec2}

In this section we review band generators of the braid group $B_n$ and 
relate canonical factors of $B_n$ to non-crossing partition diagrams of an $n$-punctured disk $D_n$.

\subsection{Band Generators}
 In the traditional presentation of $B_n,$ Artin generators only swap adjacent strands.
In the band generator presentation, any two strands are now allowed to swap. For each $i,j$ with $1\leq i<j\leq n$, define
\begin{equation*}
    a_{ij}=(\sigma_{j-1}\sigma_{j-2}\cdots\sigma_{i+1})\sigma_i(\sigma_{j-1}\cdots\sigma_{j-2}\sigma_{i+1})^{-1}.
\end{equation*}
Here, $a_{ij}$ is the element in $B_n$ which swaps the $i$th and $j$th strand, while leaving all other strands fixed. In addition, the two strands being interchanged must lie in front of all fixed strands. This gives us the band generator presentation (or Birman-Ko-Lee presentation) \cite{Birman} as follows:
\begin{equation}\label{eq:gp presentation of Bn}
B_n= \GenRels*{ a_{ij} \st
\begin{medsize}
\begin{aligned} & a_{jk}a_{ij}=a_{ij}a_{ik}=a_{ik}a_{jk},\enspace i<j<k \\[-0.5ex] %
 & a_{ij}a_{kl}=a_{kl}a_{ij}, \enspace i<j<k<l%
\end{aligned}
\end{medsize}
} 
\end{equation}
The set of band generators of $B_n$ is denoted by $\BG(B_n)$. 
Band generators $a_{i,j}$ are often called {\em positive bands} as each creates a positive crossing in the braid diagram, versus their inverses $a_{i, j}^{-1}$ are called {\em negative bands}.

In $B_4$, following the notations in \cite{KangKoLee}, the band generators are renamed as:
$$
a_1:=a_{1,2}, \,
a_2:=a_{2,3}, \,
a_3:=a_{3,4}, \,
a_4:=a_{1,4}, \,
b_1:=a_{1,3}, \,
b_2:=a_{2,4}.
$$
Visually, these six band generators are as follows. The braid strands are numbered from 1 (bottom strand) to 4 (top strand). 
\begin{figure}[h]
\centering
\includegraphics[width=11cm]{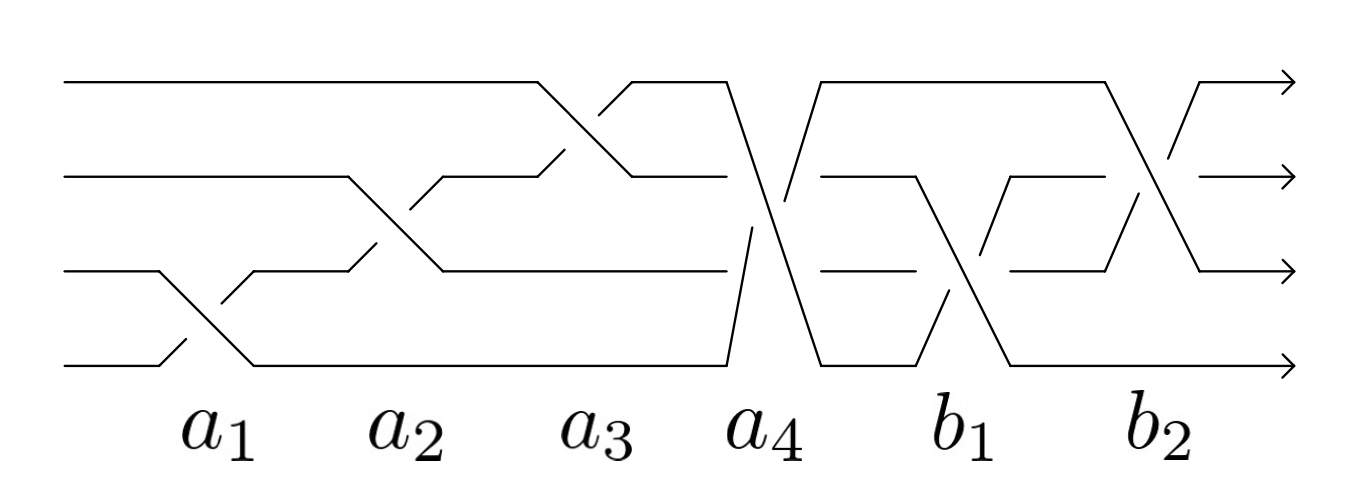}
\caption{Band generators for $B_4$}
\end{figure}
The band generators $a_1, \dots, a_4, b_1, b_2$ are positive bands and  their inverses $a_1^{-1}, \dots, a_4^{-1}, b_1^{-1}, b_2^{-1}$ are negative bands.

\subsection{Non-Crossing Partition Diagrams}
\label{sec:non-crossing partition diagram}
In this subsection, we review non-crossing partition diagrams. 
They have been introduced in the book \cite{OrderingBraids} by Dehornoy, Dynnikov,  Rolfsen, and Wiest, and Calvez and Wiest \cite{CalvezWiest} use diagrams to study 4-braids. 

Recall that $B_4$ can be viewed as the mapping class group of the four times punctured disk $D_4$. We will parameterize $D_4$ as the unit disk in $\mathbb{C}$ with punctures at
\begin{align*}
    P_1&=\frac{1}{2}e^{-\frac{3\pi i}{4}},\quad P_2=\frac{1}{2}e^{-\frac{\pi i}{4}},\\
     P_3&=\frac{1}{2}e^{-\frac{7\pi i}{4}},\quad P_4=\frac{1}{2}e^{-\frac{5\pi i}{4}}.
\end{align*}
\begin{align*}
\begin{tikzpicture}
\filldraw[fill=none, very thick](-1,0) circle (1.5);
\coordinate (1) at (-0.5,0.4);
\coordinate (2) at (-0.5,-0.6);
\coordinate (3) at (-1.5,0.4);
\coordinate (4) at (-1.5,-0.6);
\foreach \i/\Position in {1/below, 2/below, 3/above, 4/above}{
    \fill (\i) circle (1pt) node [\Position] {};}
    \node at (1) [above = 1mm of 1] {$P_3$};
    \node at (2) [above = 1mm of 2] {$P_2$};
    \node at (3) [above = 1mm of 3] {$P_4$};
    \node at (4) [above = 1mm of 4] {$P_1$};
\end{tikzpicture}
\end{align*}
Note that we placed the punctures in a counterclockwise direction, while in \cite{OrderingBraids} and \cite{CalvezWiest} punctures are placed clockwise.

Similarly, we define an $n$ punctured disk $D_n$. Punctures are labeled $P_1,\dots,P_n$ counterclockwise and $P_1$ and $P_n$ are separated by the half-line from the origin with angle $\pi$.

The braid strands can be thought of as $\{P_1, \dots P_n\} \times [0,1]$ in the cylinder $D_4 \times [0, 1]$ where the $k$-th strand is $\{P_k\} \times [0,1]$. 
The braid diagram can be thought as a view of the cylinder $D_n \times [0, 1] \subset \mathbb C \times \mathbb R$ from the position $(-1, 0.5)\in \mathbb C \times \mathbb R$.

The line segment $\overline{P_iP_j} \times \{0.5\}$ in the cylinder can be seen in the braid diagram as an arc over the braid strands $i+1, i+2, \dots, j-1$.  
This gives an idea to pictorially denote the positive band $a_{i,j}$ by the line segment $\overline{P_iP_j}$ in $D_n$. 
In fact under the identification of the braid group $B_n$ and the mapping class group of $n$-punctured disk $D_n$, the braid element $a_{i, j}$ corresponds to a positive half Dehn-twist (counterclockwise) along the line segment $\overline{P_iP_j}$. 

With this idea in mind, we associate a certain braid word, which will be called a {\em canonical factor} in Definition~\ref{def:CF} below, with a diagram. These diagrams have been used in the study of non-crossing partitions.

\begin{definition}
The diagram of $n$ dots (with no edges) represents the identity element $e$ of the braid group $B_n$. When $n=4$ we define:
\begin{equation*}
    (\e):=e.
\end{equation*}
\end{definition}

\begin{definition}\label{def:edge}
{\bf (Edge, 2-gon)}
 The diagram of $n$ dots with an edge connecting two points $P_i$ and $P_j$ ($i<j$) represents the band generator $a_{i,j}\in B_n$. When $n=4$ we define:
     \begin{align*}
        (\aone)&:=a_1&(\atwo):=a_2\\
        (\athree)&:=a_3&(\afour):=a_4\\
        (\bone)&:=b_1&(\btwo):=b_2.
    \end{align*}
\end{definition}

With the new notation we have:
\begin{align*}
    \BG(B_4) &=\{a_1,a_2,a_3,a_4,b_1,b_2\}\\
    &=\{(\aone),(\atwo),(\athree),(\afour),(\bone),(\btwo)\}.
\end{align*}

\begin{remark}\label{rem:2-gon} 
In the following, we will define $3$-gons, $4$-gons, $\dots$ so that an $n$-gon corresponds to a length $n-1$ word. This justifies viewing a single edge connecting two vertices (Definition~\ref{def:edge}) as a $2$-gon.
\end{remark}

\begin{definition}
{\bf (Triangle, 3-gon)}
Consider a triangle with vertices $P_i,P_j,P_k\in D_n$. Denote the band generator corresponding to each edge clockwise by $\alpha_0,\alpha_1,\alpha_2\in\BG(B_n)$. 
For example, in the triangle contained in the disk $D_4$, the three edges are 
$\alpha_0=a_1=$(\aone), 
$\alpha_1=b_1=$(\bone), and 
$\alpha_2=a_2=$(\atwo).
\begin{figure}[h]
\includegraphics[width=3cm]{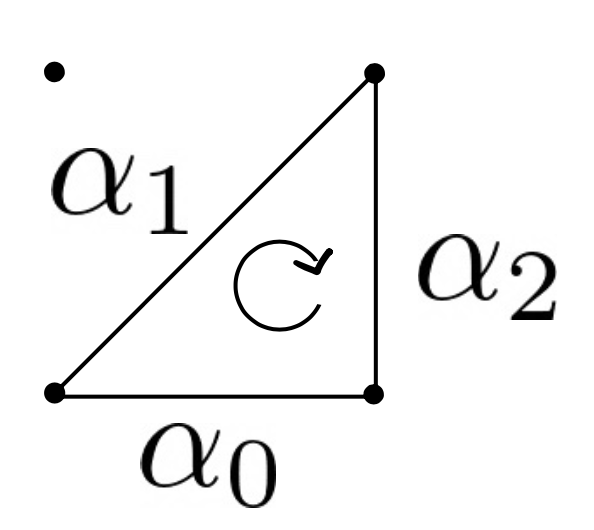}
\centering
\end{figure}

\noindent By the second relation for the band generator presentation, Eq.(\ref{eq:gp presentation of Bn}), for $B_n$, we know that $\alpha_0\alpha_1=\alpha_1\alpha_2=\alpha_2\alpha_0$ holds. Therefore, we let the triangle represent the length 2 braid $\alpha_0\alpha_1=\alpha_1\alpha_2=\alpha_2\alpha_0$. 

For $B_4$, we have four different triangles where each admits three band generator factorizations:
 \begin{align*}
         (\atwoaone)&=a_2a_1=(\atwo)(\aone)=(\aone)(\bone)=(\bone)(\atwo)\\
        (\athreeatwo)&=a_3a_2=(\athree)(\atwo)=(\atwo)(\btwo)=(\btwo)(\athree)\\
        (\afourathree)&=a_4a_3=(\afour)(\athree)=(\athree)(\bone)=(\bone)(\afour)\\
        (\aoneafour)&=a_1a_4=(\aone)(\afour)=(\afour)(\btwo)=(\btwo)(\aone).
    \end{align*}
\end{definition}

\begin{definition}
{\bf (Square, 4 gon)}
Consider a square with vertices $P_i,P_j,P_k,P_\ell\in D_n$. Denote the band generator corresponding to each edge clockwise by $\alpha_0,\alpha_1,\alpha_2,\alpha_3\in\BG(B_n)$. By the band generator presentation for $B_n$, we know $\alpha_0\alpha_1\alpha_2=\alpha_1\alpha_2\alpha_3=\alpha_2\alpha_3\alpha_0=\alpha_3\alpha_0\alpha_1$. We may define the square to represent the length 3 braid element. 
    
    In $B_4$, there is a unique square:
\begin{figure}[h]
\includegraphics[width=3cm]{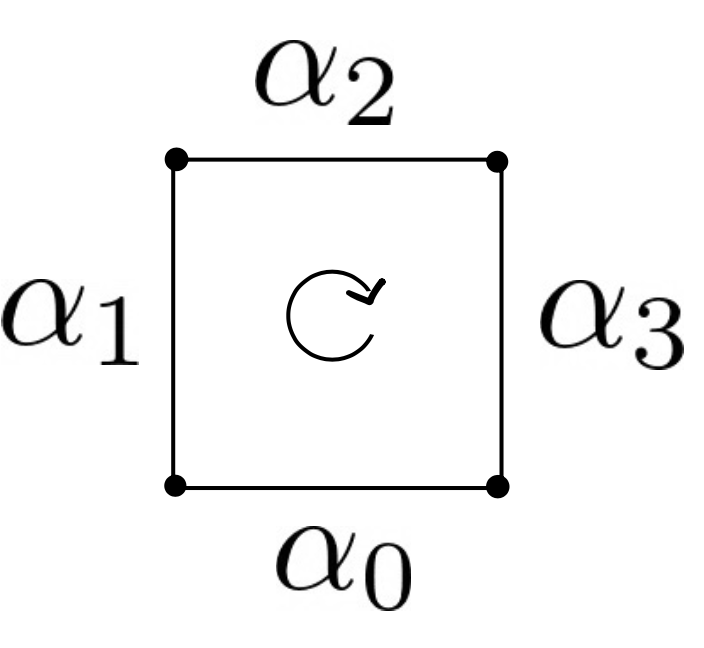}
\centering
\end{figure}

\noindent This square coincides with the   {\em fundamental element} $\delta:=a_3a_2a_1$ of $B_4$. 
It is known \cite{KangKoLee} that $\delta$ admits 12 different factorizations in $\BG(B_4)$ as follows:
\begin{align}\label{delta-expression}
a_3a_2a_1\enspace\enspace a_4a_3a_2\enspace\enspace a_1a_4a_3\enspace\enspace a_2a_1a_4\enspace\enspace b_1a_2a_4\enspace\enspace a_1b_1a_4\\
a_3b_1a_2\enspace\enspace a_1a_3b_1\enspace\enspace b_2a_1a_3\enspace\enspace a_2b_2a_1\enspace\enspace a_4b_2a_3\enspace\enspace a_2a_4b_2 \nonumber
\end{align}
Note that $\delta^4=\Delta^2$ a full twist, thus $\delta^4$ generates the center of the braid group $B_4$.
\end{definition}

We can generalize the above construction to a $k$-gon. 
\begin{definition}
{\bf ($k$-gon)}
Let $2 \leq k\leq n$. A $k$-gon in $D_n$, whose vertices oriented clockwise, represents a length $k-1$ word reading consecutive $k-1$ edges of the $k$-gon clockwise. 
\end{definition}

\begin{definition}
A disjoint union of polygons represents the product of braids corresponding to the polygons. 

For $B_4$, we have two such diagrams that represent $a_1a_3$ and $a_2a_4$. 
    \begin{align*}
        (\aoneathree)&:=a_1a_3\\
        (\atwoafour)&:=a_2a_4.
    \end{align*}
By the disjointedness condition, each of these diagrams admits two factorizations in $\BG(B_4)$ as follows:
    \begin{align*}
        (\aoneathree)&=(\aone)(\athree)=(\athree)(\aone)\\
        (\atwoafour)&=(\atwo)(\afour)=(\afour)(\atwo).
    \end{align*}
\end{definition}

\subsection{Canonical Factors}
In this section, we define canonical factors that play an essential role in the left canonical form of a braid.

An $n$-braid is {\em positive} if it is represented by a word in positive band generators. 
Let $B_n^+$ denote the monoid of positive $n$-braids.

\begin{definition}\label{def:CF}
For two braid words $V,W\in B_n$, we say $V\leq W$ if $W=PVQ$ for some (possibly empty) positive words $P,Q \in B_n^+$. Elements in the set $\{W\in B_n\mid e\leq W\leq \delta\}$ are called {\em canonical factors} and the set is denoted by $\CFn$. 
\end{definition}

In \cite[Corollary 3.5]{Birman}, Birman, Ko, and Lee showed that the cardinality of the set $\CFn$ is the $n$th Catalan number $\mathcal{C}_n=(2n)!/n!(n+1)!$. In their proof, the following theorem is implicit.

\begin{theorem}\label{thm:1-1correspondence} 
The set $\CFn$ is in a one-to-one correspondence with the set of noncrossing partitions of $n$ elements. 
\end{theorem}

For example, there are 14 canonical factors in $B_4$.  
\begin{align*}
\CF&=\{
e, a_1, a_2, a_3, a_4, b_1, b_2, a_2a_1, a_3a_2, a_4a_3, a_1a_4, a_1a_3, a_2a_4, \delta\}\\
&=\{(\e),(\aone),(\atwo),(\athree),(\afour),(\bone),(\btwo),\\&\enspace\enspace\enspace\enspace(\atwoaone),(\athreeatwo),(\afourathree),(\aoneafour),
    (\aoneathree),(\atwoafour),(\delt)\}.
\end{align*}

\section{Partial Ordering on $\CFn$}\label{Sec:3}
In this section we will discuss a natural partial ordering on the set $\CFn$ arising from the diagrams.

\begin{definition}
    The {\em set of convex hulls} of a canonical factor $A$, denoted by $\ch(A)$, is the union of the diagram $A$ and the region(s) enclosed by the edges.
\end{definition}

\begin{example}
The set of convex hulls of a triangle $\ch(\atwoaone)$ consists of the 2-dimensional triangle and the area inside the triangle enclosed by the three edges. The set of convex hulls $\ch(\aoneathree)$ only includes the two disconnected edges. We do not include the area between the two edges, as they do not enclose a connected region. 
\end{example}

Similar to the partial ordering $\leq$ on $B_n$ (Definition~\ref{def:CF}), we define a partial ordering $\prec$ on the set $\CFn$. 

\begin{definition}
For canonical factors $A, B \in \CFn$, we say $A\prec B$ if $AQ=B$ for some canonical factor $Q \in \CFn$. 
\end{definition}

We note that $A\prec B$ implies $A \leq B$.

The next lemma gives a diagrammatic interpretation of the partial ordering.

\begin{theorem}\label{thm:subset}
    Let $A,B\in\CFn$ be canonical factors. We have $\ch(A) \subset \ch(B)$ if and only if $A\prec B$.
\end{theorem}

The theorem will be proved in Propositions~\ref{prop:only-if-direction} and \ref{prop:if-direction}. 

Using Theorem~\ref{thm:subset}, we can easily obtain the Hasse diagram (Figure~\ref{fig:HasseDiagram}) for the partial ordering $\prec$ on the set $\CF$. 
Moreover we notice that $\prec$ and $\leq$ are equivalent partial orderings on $\CF$. 

 \begin{figure}[h]
 \centering
\includegraphics[width=8cm]{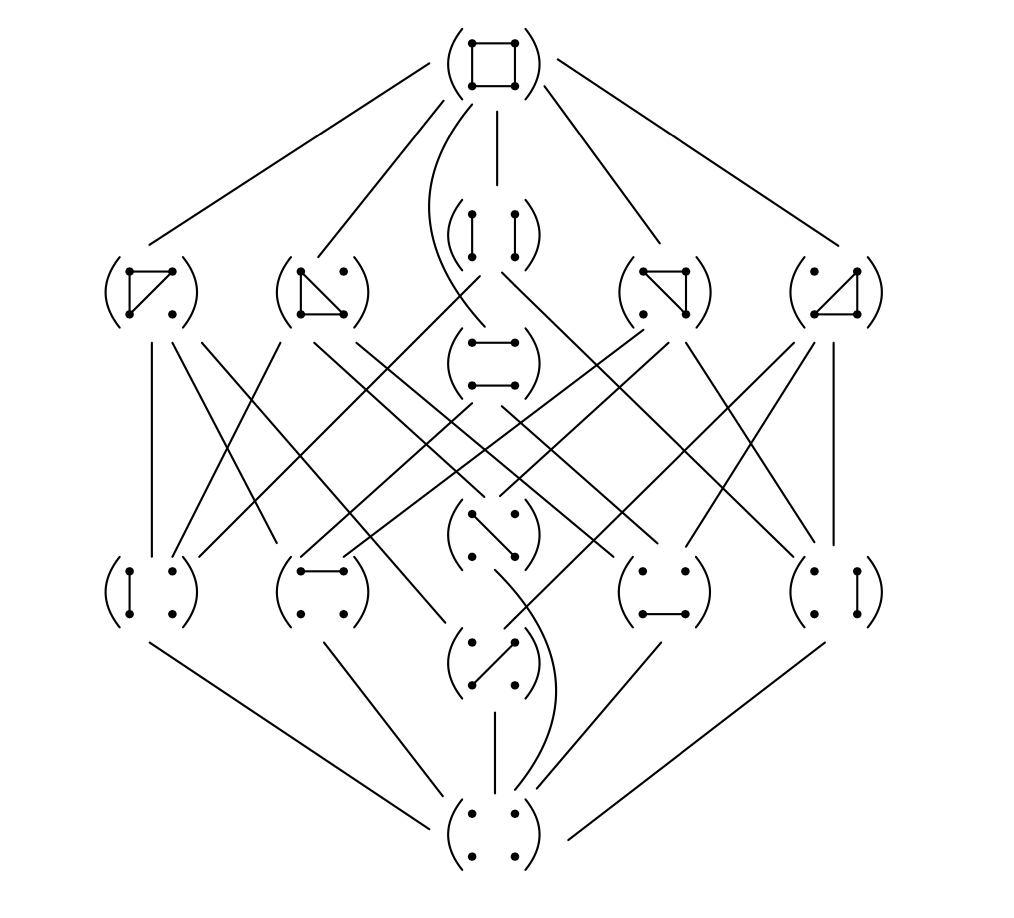}
 \caption{Hasse Diagram for the partially ordered set $(\CF, \prec)$} 
\label{fig:HasseDiagram}
\end{figure}

\subsection{Operations $\diamond$ and $\ast$ on $\CFn$.}
\label{sec:operations} 

To prove Theorem~\ref{thm:subset} we will introduce operations $\sqcup$, $\diamond$ and $\ast$ on $\CFn$.  
We use the convention that a canonical factor is denoted with a roman font (eg. $A$) and its corresponding diagram is denoted by the same alphabet in calligraphy style font (eg. $\A$).  

\begin{definition}
Let $A$ and $B$ be canonical factors whose diagrams do not intersect; $\A \cap \B = \emptyset$. The product $AB$ in the braid group $B_n$ is denoted by $A \sqcup B$ and the corresponding diagram is denoted by $\A \sqcup \B$. Since $AB=BA$ we have $A \sqcup B= B\sqcup A$.  
\end{definition}

\begin{definition}\label{def:facing}
(See Figure~\ref{fig:FacingPair}) 
Let $A\in \CFn$ be a canonical factor whose diagram $\mathcal A$ is disjoint union of polygon components. 
We say that components $\X$ and $\Y$ of $\A$ are {\em facing} to each other if there exist unique edges $x_0 \subset \X$ and $y_0 \subset \Y$ such that no other components of $\mathcal A$ lie between $x_0$ and $y_0$. 
The pair of edges $x_0$ and $y_0$ is called a {\em facing pair}. 
\end{definition}

\begin{figure}[h]
 \centering
\includegraphics[width=10cm]{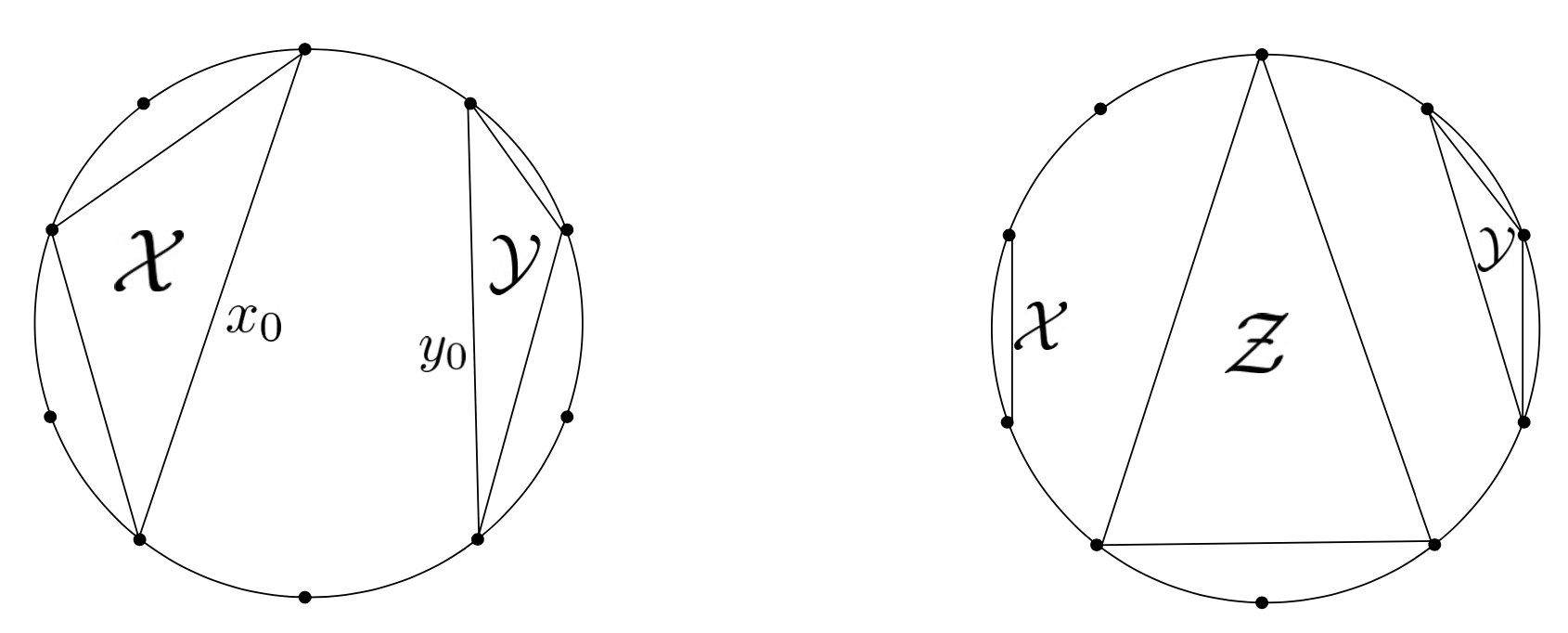}
 \caption{On the left, $A=X \sqcup Y$ and $\mathcal{X}$ and $\mathcal{Y}$ are facing each other. On the right, $A=X \sqcup Y \sqcup Z$. There exists a 3-gon $\Z$ in-between $\mathcal{X}$ and $\mathcal{Y}$, so $\X$ and $\Y$ are not facing each other.} 
\label{fig:FacingPair}
\end{figure}

The definition of the $\diamond$-operation is divided into 4 steps.

\begin{definition}\label{def:diamond}
({\bf Step 1.} See Figure~\ref{fig:Diamond}) 
Let $A, B \in \CFn$ be canonical factors such that 
\begin{enumerate}
\item
both of their diagrams $\A$ and $\B$ are connected,
\item
diagrams $\A$ and $\B$ intersect at a single vertex, say $P$, and 
\item
$\B$ lies on the left side of $\A$ near the vertex $P$ (if you stand on $p$ and see the interior of the disk in front of you). 
\end{enumerate}
Let $a_0$ and $b_0$ be the unique edges of $\A$ and $\B$ respectively such that $a_0 \cap b_0= P$ and $\A \cup \B$ lies outside the fan-shape region between $a_0$ and $b_0$. 
Assume the polygon $\A$ (resp. $\B$) has $k$ (resp. $l$) sides. 
Label the edges of $\A$ starting from $a_0$ clockwise, $a_1, \dots, a_{k-1}$.  
Label the edges of $\B$ starting from $b_0$ clockwise, $b_1, \dots, b_{l-1}$. 
Abusing the notation, each $a_i$ and $b_j$ represents a positive band generator.

When $A$ and $B$ satisfy the above conditions (and if we emphasize this fact), the usual product $AB$ in the braid group $B_n$ is denoted by $A\diamond B$. 
$$A\diamond B:= AB = a_1 \cdots a_{k-1} b_1 \cdots b_{l-1}.$$ 
\end{definition}

\begin{figure}[h]
 \centering
\includegraphics[width=10cm]{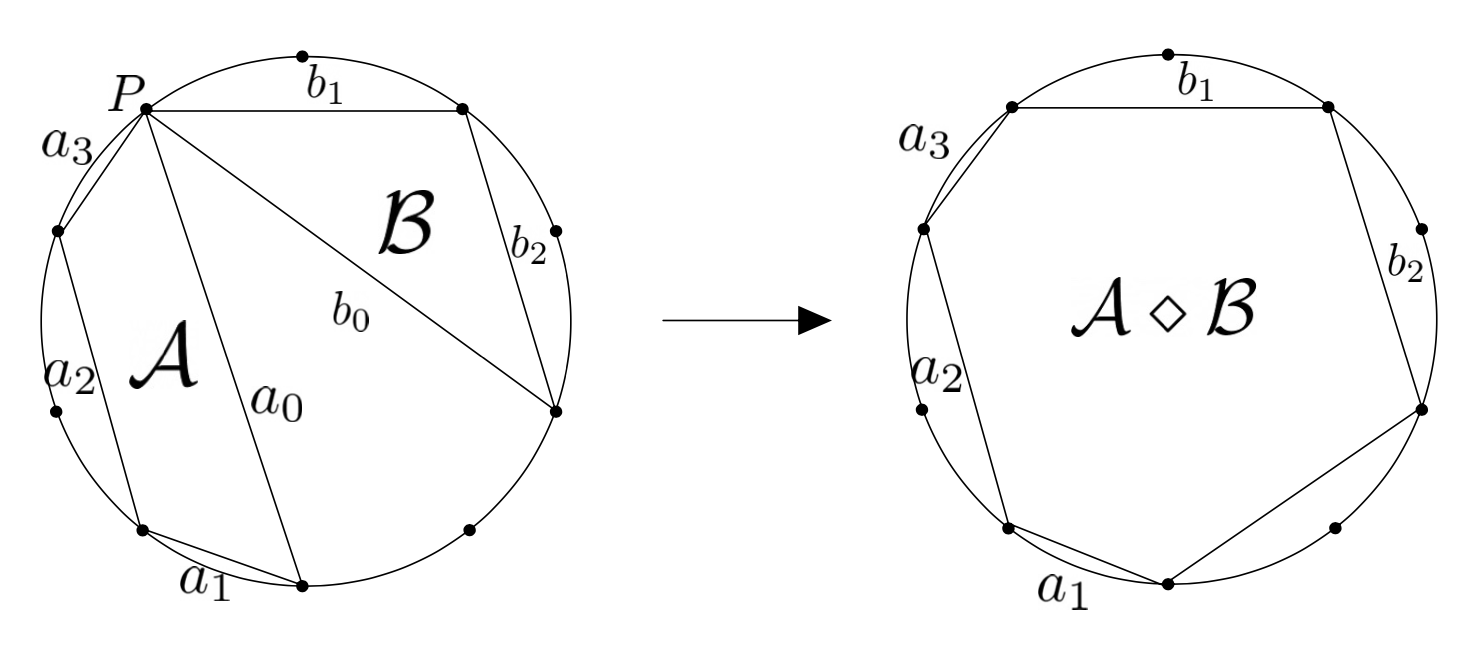}
 \caption{An example of Definition~\ref{def:diamond}, $\A\diamond \B$. The left image shows two polygons that intersect at the point $P$. Here, $\A$ is a 4-gon, $\B$ is a 3-gon, and $\A\diamond\B$ is a 6-gon.} 
\label{fig:Diamond}
\end{figure}

We list properties of $A\diamond B$: 
\begin{itemize}
\item
The diagram of $A\diamond B$, which is denoted by $\A\diamond \B$, is a single polygon with  $k+l-1$ sides. 
\item
By Theorem~\ref{thm:1-1correspondence} the product $A\diamond B$ is a canonical factor.  
\item
$
\ch(A) \cup \ch(B) \subset \ch(A\diamond B). 
$
More precisely, $\A\diamond \B$ is the minimal polygon whose convex hull contains $\A$ and $\B$.
\end{itemize} 

By the condition (3), $B \diamond A$ does not make sense, i.e., $\diamond$ is a non-commutative operation. 

A simple example of the operation $\diamond$ is: 
$$(\afour) \diamond (\athree) = (\afourathree)$$

We will extend the operation $\diamond$ to less restrictive pairs of canonical factors. 

\begin{definition}\label{def:diamond2}
({\bf Step 2.} See Figure~\ref{fig:Step2}) 
Let $A=A_1 \sqcup \cdots \sqcup A_k$ and $B\in \CFn$ be canonical factors such that 
\begin{itemize}
\item
each of the diagrams $\A_1, \dots, \A_k$ and $\B$ is a connected component.
\item
for each $i=1, \dots, k$ the pair $(A_i, B)$ satisfy the (Step 1) condition; namely, the intersection $\A_i \cap \B $ is a single vertex, say $P_{A_i}$, of the disk $D_n$, and 
\item
in a small neighborhood of $P_{A_i}$, the diagram $\B$ lies on the left of $\A_i$.
\end{itemize}
We extend the operation $\diamond$ to the pair $A$ and $B$ by
$$
A\diamond B := A_1 \diamond (A_2 \diamond ( \cdots \diamond (A_k \diamond B)))
$$
where the $\diamond$ in the right hand side is defined in (Step 2) in Definition~\ref{def:diamond}. 
\end{definition}

\begin{figure}[h]
 \centering
\includegraphics[width=12cm]{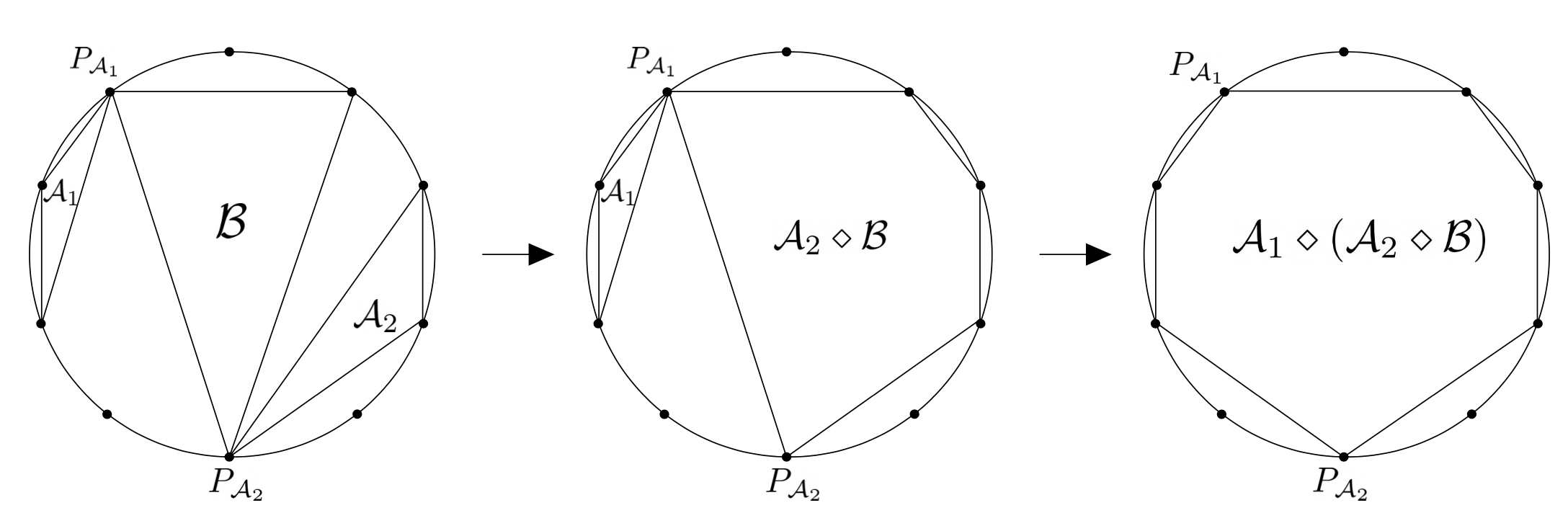}
 \caption{An example of Definition~\ref{def:diamond2}. We have $\A=\A_1\sqcup \A_2$ and a 3-gon $\B$. We apply Definition~\ref{def:diamond2} to obtain $\A\diamond\B=\A_1\diamond(\A_2\diamond\B)$.}
\label{fig:Step2}
\end{figure}

We remark that $\diamond$ is well defined since for every permutation $\sigma \in S_k$ of $k$ elements, 
$$
A_1 \diamond (A_2 \diamond ( \cdots \diamond (A_k \diamond B))) =
A_{\sigma(1)} \diamond (A_{\sigma(2)} \diamond ( \cdots \diamond (A_{\sigma (k)} \diamond B))).
$$

\begin{proposition}\label{prop:list of properties of diamond}
We list properties of the operation $\diamond$ below: 
\begin{itemize}
\item[{\rm (a)}]
As braids, $A\diamond B = AB$. 
\item[{\rm (b)}]
The corresponding diagram of $A\diamond B$, denoted by $\A \diamond \B$, is the minimal disjoint union of polygons whose convex hull contains $\ch(A)$ and $\ch(B)$.  
\item[{\rm (c)}]
In particular, $\ch(A) \cup \ch(B) \subset \ch(A\diamond B)$. 
\item[{\rm (d)}]
$A\diamond B \in\CFn$ by Theorem~\ref{thm:1-1correspondence}. 
\end{itemize} 
\end{proposition}

For Step 2, with regard to {\rm (b)} the diagram $\A\diamond\B$ is a single polygon.

\begin{definition}\label{def:diamond3}
({\bf Step 3.} See Figure~\ref{fig:Step3})
Let $A = A_1 \sqcup \cdots \sqcup A_k \sqcup A'$ and $B\in \CFn$ be canonical factors such that 
\begin{itemize}
\item
the pair $(A_1\sqcup \dots\sqcup A_k, B)$ satisfy the (Step 2) condition for Definition~\ref{def:diamond2},
\item
and $\A' \cap \B =\emptyset$. 
\end{itemize} 
We define 
$$A \diamond B := A' \sqcup ((A_1 \sqcup \cdots \sqcup A_k) \diamond B)
$$
where $\diamond$ on the right hand side is in the sense of (Step 2). 
The disjoint operation in the right hand side between $A'$ and $(A_1 \sqcup \cdots \sqcup A_k) \diamond B$ is justified by the above property $(b)$ and the disjointness $\A' \cap \B =\emptyset$. 
The diagram $\A\diamond \B$ is the disjoint union of $\A'$ and the polygon $(\A_1 \sqcup \cdots \sqcup \A_k) \diamond \B$. 
Therefore, above properties (a), (b), (c), and (d) also hold. 
\end{definition}

\begin{figure}[h]
 \centering
\includegraphics[width=10cm]{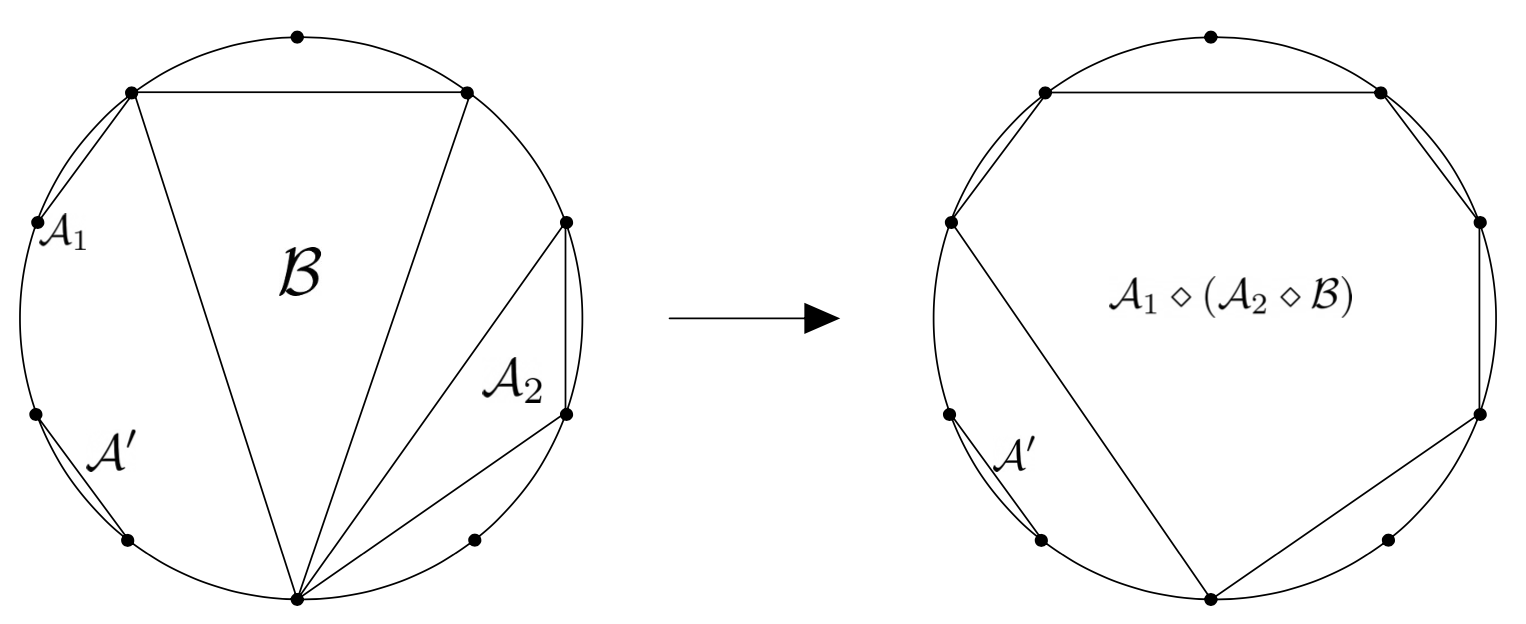}
\caption{An example of Definition~\ref{def:diamond3}. We have $\A=\A_1\sqcup \A_2\sqcup \A'$ and a 3-gon $\B$. We apply Definition~\ref{def:diamond3} to obtain $\A\diamond \B=\A'\sqcup((\A_1\sqcup \A_2)\diamond \B)$.}
\label{fig:Step3}
\end{figure}

\begin{definition}\label{def:diamond4}
({\bf Step 4}) 
Let $A, B \in \CFn$ such that 
\begin{itemize}
\item
$\A \cap \B = \ch(A) \cap \ch(B) \subset \{ \mbox{vertices of } D_n\}$ and 
\item
if $P\in \A\cap\B$, near $P$ the diagram $\B$ lies on the left of $\A$. 
\end{itemize}
Therefore, $\A$ and $\B$ do not share any edges or points away from the vertices.
It is possible that $\A \cap \B = \emptyset$.

The factor may admit a decomposition $B = B_1 \sqcup \cdots \sqcup B_l \sqcup B' \in \CFn$ such that for each $i=1, \dots, l,$ the pair $(A, B_i)$ satisfy (Step 3) condition in Definition~\ref{def:diamond3} and $\A \cap \B' = \emptyset.$

We define
$$A \diamond B := B' \sqcup ((((A \diamond B_1) \diamond B_2 ) \diamond \cdots ) \diamond B_l)
$$
where the $\diamond$ in the right hand side is in the sense of  Definition~\ref{def:diamond3}. 
The operation $\diamond$ is independent of the ordering of factors $B_i$s. 
The above properties (a), (c) and (d) are also satisfied. 
\end{definition} 

Next we define the operation $*$. 

\begin{definition}\label{def:operation-ast}
(See Figure~\ref{fig:Star}) 
Let $A, B \in \CFn$ be canonical factors such that 
\begin{itemize}
\item 
both of their diagrams $\A$ and $\B$ are connected, and
\item
their diagrams do not intersect $\A \cap \B = \emptyset$ (i.e., $AB=BA$ in $B_n$). 
\end{itemize}
Suppose that the polygon $\A$ (resp. $\B$) has $k$ (resp. $l$) sides. 
Let $a_0$ and $b_0$ be edges of $\A$ and $\B$ that form the facing pair.  
Edges of $\A$ and $\B$ are named $a_0,\dots,a_{k-1}$ and $b_0,\dots, b_{l-1}$ labeled clockwise. 
Thus
$A=a_1 \cdots a_{k-1}$ and $B=b_1\cdots b_{l-1}$. 
Denote by $P_A$ the vertex of $\A$ where the edges $a_0$ and $a_{k-1}$ meet. 
Similarly, denote by $P_B$ the vertex of $\B$ where the edges $b_0$ and $b_{l-1}$ meet.
Let $\C$ denote the edge joining the vertices $P_A$ and $P_B$. 
We call $\C$ the {\em joining edge} of $\A$ and $\B$. 
The joining edge $\C$ corresponds to a positive band generator, denoted by $C \in \BG(B_n)$ and can be thought as a bigon (cf. Remark~\ref{rem:2-gon}.) 
Define an operation $\ast$ by 
$$
A\ast B := A\diamond (B\diamond C)=B\diamond (A\diamond C)=(A\sqcup B)\diamond C.
$$
As braids, $A\ast B = ABC=BAC$ in the braid group $B_n$.  
\end{definition}

\begin{figure}[h]
 \centering
\includegraphics[width=12cm]{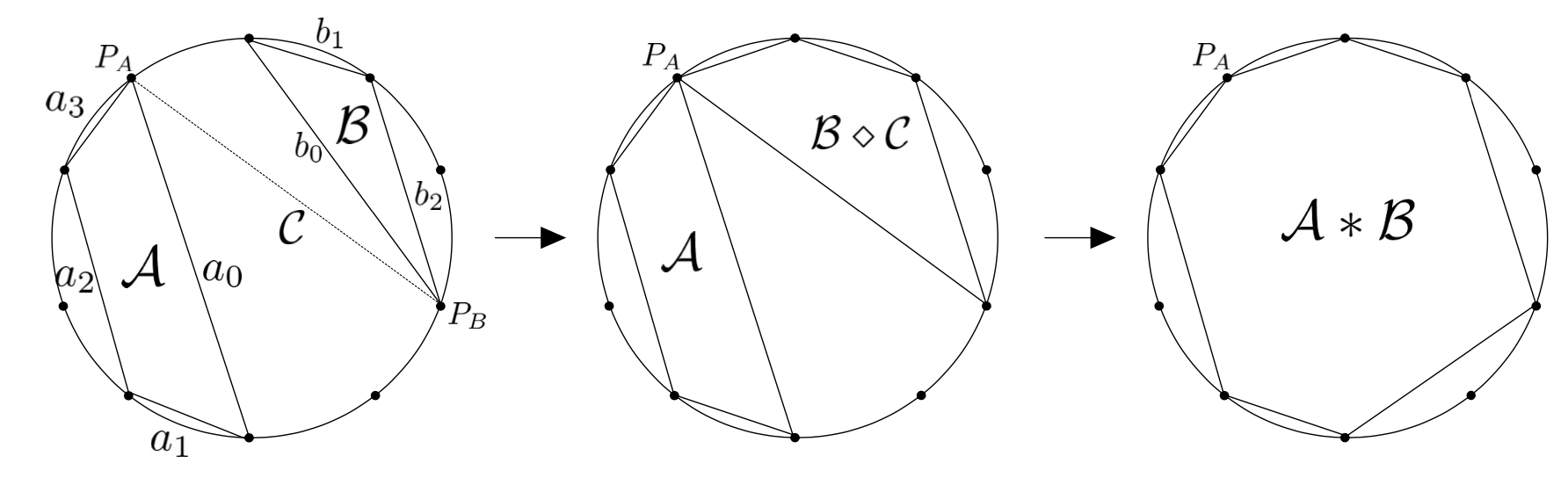}
 \caption{An example of the operation $\A\ast\B$. The left image shows two polygons that do not intersect. $\A$ is a 4-gon, $\B$ is a 3-gon, and $\C$ is the joining edge. The middle image shows the computation of $\B\diamond\C$. The right image shows the computation of $\A\ast\B$.} 
\label{fig:Star}
\end{figure}

\begin{remark}\label{remark:joining edge}
For later use, we note that in neighborhood of $P_A$ the joining edge $\C$ lies on {\em left} of $\ch(A)$. Likewise, near $P_B$ the edge $\C$ lies on  {\em left} of $\ch(B)$.   
\end{remark}

We list properties of the operation $*$ immediately follow from the definition. 

\begin{proposition}\label{prop:property of ast}
The operation $\ast$ is commutative:
$A\ast B = B \ast A.$
The diagram of $A\ast B$, denoted by $\A \ast \B$, satisfies 
$$
\A\ast\B=\A \diamond (\B \diamond \C)=\B \diamond (\A \diamond \C)=(\A \sqcup \B) \diamond \C
$$ 
and it is a single polygon with $k+l$ sides.  
Thus, by Theorem~\ref{thm:1-1correspondence}, we see that $A\ast B\in \CFn$.
The definition also implies
$$
\ch(A) \sqcup \ch(B) \subset \ch(A \ast B).
$$
\end{proposition}

\subsection{Proof of Theorem~\ref{thm:subset}}
\label{sec:proof of thm}

\begin{lemma}\label{lem:A'} {\rm(See Figure~\ref{fig:LemA})}
Let $A, B \in \CFn$ be canonical factors such that 
\begin{itemize}
\item
the diagram $\B$ is connected, 
\item
the diagram $\A$ has $k(>1)$ connected components, and 
\item
$\ch(A) \subset \ch(B).$
\end{itemize} 
Then there exists a band generator $C \in \BG(B_n)$ such that such that $A':=AC$ satisfies
\begin{itemize}
\item 
$A'=A\diamond C$, 
\item
$A'$ is a canonical factor, 
\item
its diagram $\A'$ has $k-1$ connected components, and 
\item
$\ch(A) \cup \C \subset \ch(A') \subset \ch(B).$
\end{itemize}
\end{lemma}

\begin{proof}
The diagram $\A$ contains a disjoint pair of connected components facing to each other. Call them $\X$ and $\Y$ and their corresponding canonical factors are denoted by $X$ and $Y$ respectively. 
Therefore, $A = X \sqcup Y \sqcup Z$ for some $Z \in \CFn$. 
Denote the joining edge of $\X$ and $\Y$ by $\mathcal C$. 
Let $C\in \BG(B_n)$ denote the corresponding positive band generator for $\mathcal C$.
Note that the pair $(A, C)$ satisfy the (Step 3) condition of $\diamond$. 
We may define
$$A':=A\diamond C=(X \ast Y) \sqcup Z$$
Its diagram $\A'$ is exactly $\A$ except for the two components $\X$ and $\Y$ are replaced by the single polygon $\X \ast \Y=\X\diamond (\Y\diamond \mathcal C)$. 
Therefore, $A'$ is a canonical factor by Theorem~\ref{thm:1-1correspondence} and 
$\A'$ consists of $n-1$ components. By the nature of the operation $\ast$ we have the inclusions $\ch(A) \subset \ch(A') \subset \ch(B).$
\end{proof}

\begin{figure}[h]
 \centering
\includegraphics[width=10cm]{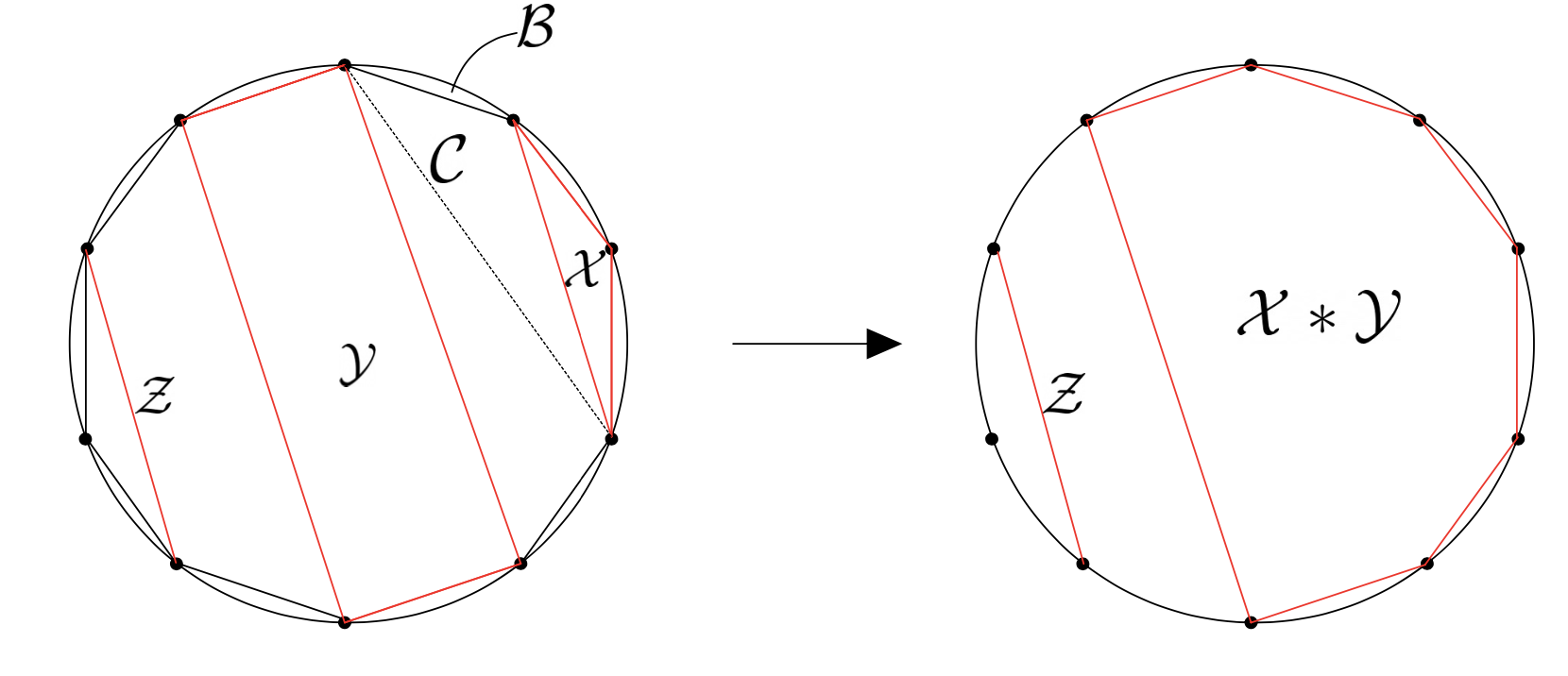}
 \caption{The above figure is an example of Lemma \ref{lem:A}. The left image contains a 10-gon $\B$ and a 3-component diagram $\A = \X \sqcup \Y \sqcup \Z$ in red. On the right image, we compute $\A'=(\X \ast \Y) \sqcup \Z$, which is a 2-component diagram in red.} 
\label{fig:LemA}
\end{figure}

\begin{corollary}\label{cor:A'}
{\rm(See Figures~\ref{fig:CorA} \& \ref{fig:CorA2})} Let $A, B \in \CFn$ be canonical factors such that 
\begin{itemize}
\item
the diagram $\B$ is connected, 
\item
the diagram $A = A_1 \sqcup\cdots\sqcup A_k$ where $k>1$ and $\A_i$ is connected, and 
\item
$\ch(A) \subset \ch(B).$
\end{itemize} 
Then there exists a canonical factor $C \in \CFn$ such that $A':=AC$ satisfies 
\begin{itemize}
\item
$A'=A\diamond C$,
\item
$A'$ is a canonical factor, 
\item
its diagram $\A'$ is a single polygon, and 
\item
$\ch(A) \subset \ch(A') \subset \ch(B).$
\end{itemize}
Moreover, up to reordering of the disjoint components, $$A' = (((A_1\ast A_2) \ast A_3) \ast \cdots A_{k-1})\ast A_k.$$
\end{corollary}
\begin{figure}[h]
 \centering
\includegraphics[width=10cm]{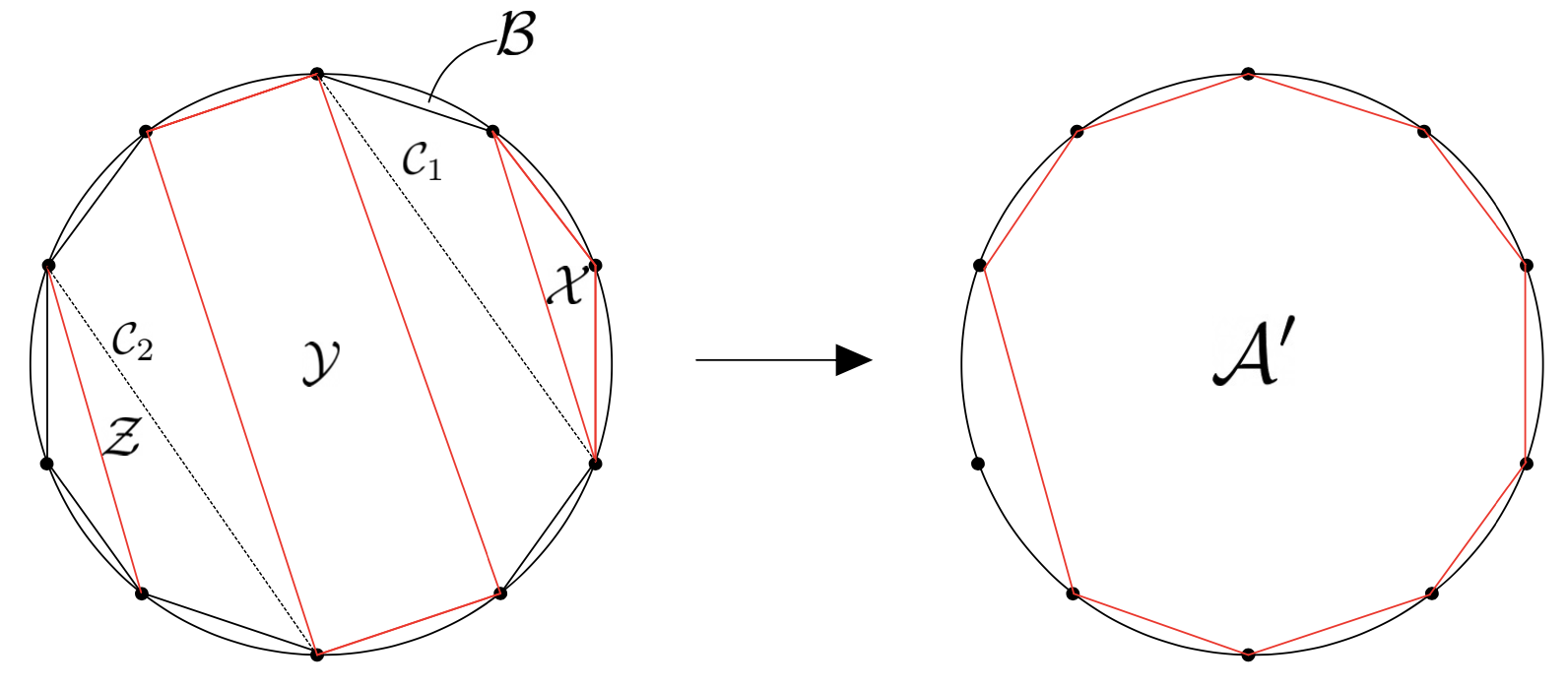}
 \caption{An example of Corollary~\ref{cor:A'}. The left image contains a 10-gon $\B$, a 3-component diagram $\mathcal{A}=\X \sqcup \Y \sqcup \Z$ in red, and $\C= \C_1 \sqcup \C_2$. On the right image, we compute $\A'$, which is a single polygon in red.} 
\label{fig:CorA}
\end{figure}
\begin{figure}[h]
 \centering
\includegraphics[width=10cm]{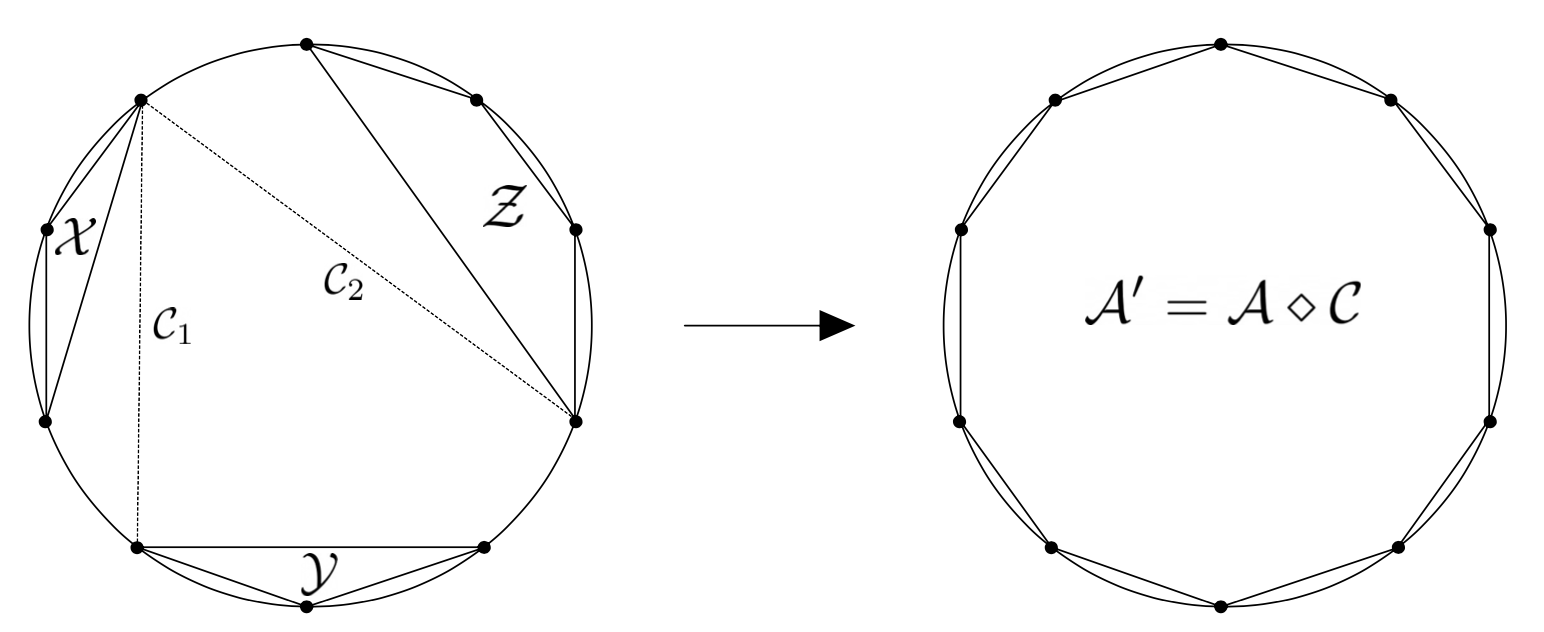}
 \caption{Another example of Corollary~\ref{cor:A'}. The left image contains a 3-component diagram $\A=\X\sqcup\Y\sqcup\Z$. On the right image, we compute $\A'=\A\diamond C$, which is a single polygon.} 
\label{fig:CorA2}
\end{figure}
\begin{proof}
Applying Lemma~\ref{lem:A'} $k-1$ times, we inductively obtain band generators  $C_1, \cdots, C_{k-1} \in \BG(B_n)$ such that 
\begin{eqnarray*}
A'
&:=&
(\cdots((A\diamond C_1) \diamond C_2)\diamond \cdots) \diamond C_{k-1}\\
&=&
A C_1 C_2 \cdots C_{k-1}
\end{eqnarray*}
is a canonical factor, the diagram $\A'$ is connected, and $\ch(A)\subset \ch(A') \subset \ch(B).$ Therefore, it is enough to show the product $C:=C_1\cdots C_{k-1}$ is a canonical factor. 

By the construction in Lemma~\ref{lem:A'} we see
$$\C_1 \subset \ch(A \diamond C_1)\quad \mbox{ and } \quad \ch(A\diamond C_1) \cap \Int(\C_2) = \emptyset.$$ 
Thus interiors of the joining edges $\C_1$ and $\C_2$ are disjoint. 
Likewise for every $i=1, \dots, k-2$ we have 
$$\C_1, \cdots, \C_i \subset \ch(AC_1\cdots C_i)$$ 
and 
$$
\ch(AC_1\cdots C_i) \cap \Int(\C_{i+1})=\emptyset.
$$
This means interiors of the joining edges $\C_1, \dots, \C_k$ are pairwise disjoint. 

If $\C_i$ and $\C_j$ intersect, they only intersect at a single vertex, say $P\in D_n$. 
Moreover, Remark~\ref{remark:joining edge} about joining edges implies that: 
$i<j$ if and only if in small neighborhood of $P$ the edge $\C_j$ lies on the left of $\C_i$. 
In other words, near $P$, joining edges meet in the counterclockwise manner. 
The interior disjointness property and the counterclockwise property guarantee that
$$C= C_1\cdots C_{k-1}= (((C_1 \diamond C_2) \diamond C_3) \diamond \cdots )\diamond C_{k-1}$$ 
and its diagram is disjoint union of polygons, that is, by Theorem~\ref{thm:1-1correspondence}, $C$ is a canonical factor. 
Since the pair $(A, C)$ satisfy the (Step 4) condition of $\diamond$ in Definition~\ref{def:diamond4} we have $A' = A \diamond C$. 
\end{proof}

We note that 
the interior disjointness property and the counterclockwise property have already appeared in the setting of Definition~\ref{def:diamond}.

\begin{lemma}\label{lem:B=AC}
{\rm(See Figure~\ref{fig:LemB=AC})} Let $A, B \in \CFn$ be canonical factors such that 
\begin{itemize}
\item
the diagram $\B$ is connected, 
\item
the diagram $\A$ is connected or empty, and 
\item
$\ch(A) \subset \ch(B).$
\end{itemize} 
Then there exists a canonical factor $C \in \CFn$ such that $$B=A\diamond C=AC.$$ 
\end{lemma}

\begin{figure}[h]
 \centering
\includegraphics[width=10cm]{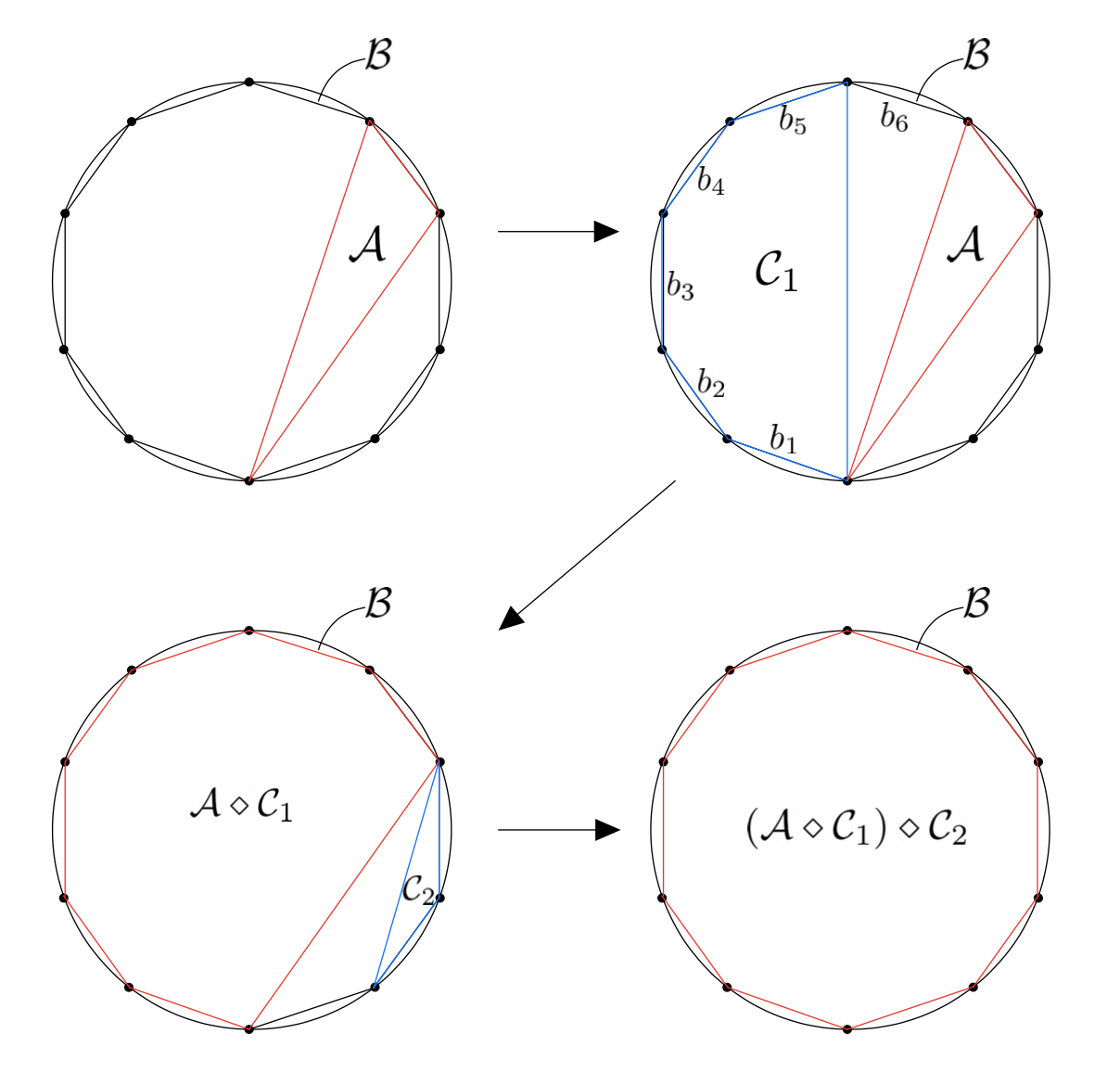}
 \caption{The above figure is an example of Lemma \ref{lem:B=AC}. The top left image contains a 10-gon $\B$ and a 3-gon $\mathcal{A}$ in red. On the top right image, we draw $\C_1$ in blue. On the bottom left image, we compute $\A\diamond\C_1$. We also draw $\C_2$ in blue. On the bottom right, we compute $\B=(\A\diamond \C_1)\diamond \C_2$}. 
\label{fig:LemB=AC}
\end{figure}

\begin{proof}
If $\A$ is empty, then $A=e$ and we set $C:=B$, done. 
In the following we assume $\A$ is non-empty and connected. 

The set $\B \setminus (\A \cap \B)$ consists of disjoint arcs.
Let $k$ denote the number of these disjoint arcs.  
If $k=0$ then there is nothing to argue since $A=B$. 
Assume that $k>0$. 
Pick one of the arcs. Let $m$ be the number of edges to form the arc. 
Orient the polygon $\B$ clockwise. 
Following the orientation, denote the $m$ edges of the arc by $b_1, \dots, b_m$. 
The product $C_1:=b_1\cdots b_{m-1}$ is a canonical factor whose diagram $\C_1$ is an $m$-gon.
We note that the pair $(A, C_1)$ satisfy the (Step 1) condition for the operation $\diamond$.
By Definition ~\ref{def:diamond}, we see that 
\begin{itemize}
\item
the product $AC_1= A \diamond C_1$ is a canonical factor
\item 
its diagram $\A\diamond \C_1$ is connected, 
\item
the number of arc components of $\B\setminus ((\A\diamond\C_1) \cap \B)$ is $k-1$, and  
\item
$\ch(A) \subset \ch(AC_1) \subset \ch(B)$. 
\end{itemize}
Thus, we can repeat the procedure for $AC_1$ and $B$ to obtain a canonical factor $AC_1 C_2$ such that 
$$\ch(AC_1) \subset \ch(AC_1C_2) \subset \ch(B).$$ 

We continue the procedure until we exhaust all the original $k$ arcs of $\B \setminus (\A \cap \B)$. 
At the end, we have canonical factors $C_1, \cdots, C_k \in \CFn$ such that 
$$AC_1\cdots C_k = (((A\diamond C_1)\diamond C_2) \diamond \cdots ) \diamond C_k= B.
$$
Since their diagrams $\C_1, \cdots, \C_k$ are pairwise disjoint, by Theorem~\ref{thm:1-1correspondence}, the product $C:=C_1 \cdots C_k=C_1 \sqcup \cdots \sqcup C_k$ is also a canonical factor and we have $B=AC=A\diamond C$.
\end{proof}

We prove the only-if part of Theorem~\ref{thm:subset}. 
The if part of Theorem~\ref{thm:subset} will be proved in Proposition~\ref{prop:if-direction}. 

\begin{proposition}\label{prop:only-if-direction}
Let $A,B\in\CFn$. If $\ch(A) \subset \ch(B)$ then there exists a $Q\in \CFn$ such that $B=AQ$; thus, $A\prec B$.
\end{proposition}

\begin{proof}
Suppose that $\ch(A) \subset \ch(B)$ and the diagram $\B$ has $k$ connected components $\B_1,\dots,\B_k$.
Suppose that $A$ admits a decomposition $A = A_1 \sqcup \cdots \sqcup A_k$ such that 
$\ch(A_i) \subset \ch(B_i)$ or $A_i = e$. 
The diagram $\A_i$ has possibly multiple components.

If $A_i\neq e$ applying Corollary~\ref{cor:A'} to every pair $(A_i, B_i)$, we can find a canonical factor $C_i \in \CFn$ such that 
\begin{itemize}
\item
$A_i':=A_i C_i = A_i \diamond C_i \in \CFn$, 
\item
the new diagram $\A_i'$ is connected and 
\item
$\ch(A_i)\subset \ch(A_i')\subset \ch(B_i)$. 
\item
$\ch(C_i) \subset \ch(A_i') \subset \ch(B_i)$.
\end{itemize}
If $A_i=e$ then we may set $C_i=e$.

Next we apply Lemma~\ref{lem:B=AC} to the pair $(A_i', B_i)$ for each $i=1,\dots,k$. 
Then there exists a canonical factor $D_i \in \CFn$ such that 
$B_i =A_i' \diamond D_i.$

Since components of $\B$ are pairwise disjoint, diagrams $\C_i$ are pairwise disjoint, so the product $C:= C_1 \sqcup \cdots \sqcup C_k$ is a canonical factor. Similarly, the product $D:=D_1 \sqcup \cdots \sqcup D_k$ is a canonical factor. 
With the $C$ and $D$ we have $B=(A\diamond C) \diamond D$. 

It remains to prove that the product $Q:=CD$ satisfies $CD = C \diamond D$;  thus, $Q$ is a canonical factor and $B=A \diamond Q$.

Note that 
\begin{eqnarray*}
&&\ch(C) \subset \ch(A  \diamond C) \subset \ch(B), \\
&&\ch(D) \subset \ch(B), \\
&&\Int(\ch(D)) \cap \ch(A \diamond C)=\emptyset.
\end{eqnarray*} 
This means if a component of $\C$ and a component of $\D$ intersect, then they must intersect at a vertex of $D_n$.
Moreover, near that vertex, the component of $\D$ lies on the left of the component of $\C$. 
Thus the (Step 4) condition of $\diamond$ in Definition~\ref{def:diamond} is satisfied which yields that the product $Q=CD=C\diamond D$. 
\end{proof}

\begin{corollary}\label{cor:complementary} 
For every canonical factor $A\in \CFn$ there exists a $B\in\CFn$ such that $A\diamond B=AB=\delta.$ 
\end{corollary}

\section{Partially Ordered Set $(\Wds(\beta), \Rightarrow)$}\label{Sec:OrderedSet}
Given $\beta\in B_n$, we define a set:
\begin{equation*}
  \Wds(\beta) := \left\{ W=\delta^r A_1\cdots A_k\middle\vert \begin{array}{l}
    W\text{ is a word representing }\beta, \\
    r\in\mathbb{Z}, \ A_i\in \CFn
  \end{array}\right\}
\end{equation*}

In this section, 
extending the partial ordering $\prec$ on the set $\CFn$,
we introduce a partial ordering $\Rightarrow$ on the set $\Wds(\beta)$ of braid words representing $\beta\in B_n$. 

We start with defining a relation $\Rightarrow$ between words. 

\begin{definition}\label{def:Rightarrow}
Let $A,B,A',B'\in\CFn$. We denote
    \begin{equation*}
        AB\Rightarrow A'B'
\end{equation*}
if the following are satisfied:
    \begin{enumerate}[a)]
        \item $AB=A'B'$ as braid elements in $B_n$.
        \item $A \prec A'$ (equivalently, $\ch(A)\subset \ch(A')$).
    \end{enumerate}
In other words, the word $A'B'$ is more {\em left-weighted} than $AB$. 
If no $A', B' \in \CFn$ satisfy $AB\Rightarrow A'B'$ then we say that the word $AB$ is {\em maximally left weighted.}
In case of $B'=e$, $B'$  may be omitted.
\end{definition}

Before giving examples of  $AB\Rightarrow A'B'$ we recall the right complementary set $R(A)$ and the starting set $S(B)$. 
The reason we do so here is that: The word $AB$ admits a more left-weighted word $A'B"$ if and only if $R(A)\cap S(B) \neq \emptyset$ (see Remark~\ref{rem:RandS}). 

\begin{definition}
Let $A\in \CFn$. Define the {\em right complementary set} $R(A)$ and the {\em starting set} $S(A)$ as follows:
\begin{align*}
       R(A)&=\{c\in\BG(B_n)\mid Ac\in\CFn\}\\
        S(A)&=\{c\in\BG(B_n)\mid A=cA'\text{ for some }A'\in\CFn \}.
\end{align*}
In terms of the diagram, $R(A)$ is the set of edges $\C$ which are  
\begin{itemize}
\item
disjoint from the diagram $\A$ of $A$; i.e., $\A \cap \C = \emptyset$, or 
\item
satisfying the two conditions in Definition~\ref{def:diamond4}; i.e., 
$\ch(A) \cap \C = \A \cap \C \subset \{\mbox{vertices of } D_n\}$ and in small neighborhood of each $P \in \ch(A) \cap \C$ the edge $\C$ lies on the left of $\A$.
\end{itemize}
It is easy to see that $S(A)$ is the set of all edges of the diagram $\A$.
\end{definition}

\begin{example}\label{ex:RS}
For each canonical factor $A\in \CF$ of $B_4$, the set $R(A)$ is described as follows (up to rotation):
\begin{align*}
    R(\aone)&=\{(\bone),(\athree),(\afour)\}=(\aonehigh)&R(\bone)=\{(\afour),(\atwo)\}=(\bonehigh)\\
    R(\atwoaone)&=\{(\afour)\}=(\atwoaonehigh)&R(\aoneathree)=\{(\bone)\}=(\aoneathreehigh)\\
    R(\delt)&=\{(\e)\}=(\e)
\end{align*}
The thick highlighted edges represent elements of $R(A)$, while the black edges represent $A$. 
See Example~\ref{ex:2} below for explanation of $R($\aoneathree).

The set $S(A)$ is described as follows (up to rotation):
\begin{align*}
    S(\aone)&=\{(\aone)\}&S(\bone)=\{\bone\}\\
    S(\atwoaone)&=\{(\aone),(\atwo),(\bone)\}&S(\aoneathree)=\{(\aone),(\athree)\}\\
    S(\delt)&=\{(\aone),(\atwo),(\athree),(\afour),(\bone),(\btwo)\}
\end{align*}
The last equation comes from (\ref{delta-expression}). 
\end{example}

The maximally left weighted condition (Definition~\ref{def:Rightarrow}) of word $AB$ can be interpreted in terms of the sets $R(A)$ and $S(B)$. 

\begin{remark}\label{rem:RandS}
The algebraic condition $R(A)\cap S(B)=\emptyset$ is equivalent to the pictorial condition that the diagram $B$ does not intersect the highlighted edges of $A$ (cf. Example~\ref{ex:RS}).   
Therefore, 
$R(A)\cap S(B)=\emptyset$ if and only if no factors $A, B \in \CF$ can increase the partial order $AB \Rightarrow A'B'$, equivalently the $AB$ is maximally left weighted. 
\end{remark}

Next, we will show two examples of $AB\Rightarrow A'B'$ using $R(A)$ and $S(B)$. 

\begin{example}\label{ex:1}
 This example will show that 
 \begin{equation*}
     AB:=(\bone)(\atwoafour)\Rightarrow (\afourathree)(\atwo)\Rightarrow (\delt) (\e)= (\delt).
 \end{equation*}
 We will start by highlighting $A$ as follows:
 \begin{equation*}
     (\bonehigh)(\atwoafour).
 \end{equation*}
 Note that $B$ is contained in the highlighted edges of $A$. Move one of these lines from $B$ to $A$ to make it more left weighted. 
 It does not matter which line you choose to move. Then, we obtain
 \begin{equation*}
AB\Rightarrow    A'B':= (\afourathree)(\atwo).
 \end{equation*}
 Next, we repeat this process until $R(A')\cap S(B')=\emptyset$. If we highlight $A'$, we have
 \begin{equation*}
     (\afourathreehigh)(\atwo).
 \end{equation*}
 Note that the edge of $B'$ is contained in the highlighted edge of $A'$. Move this line from $B'$ to $A'$ to make it more left weighted. 
 Then, we obtain
$$A'B'\Rightarrow A''B'':=(\delt)(\e).$$ 
 \end{example}
  
\begin{example}\label{ex:2}
This example will show
\begin{equation*}
    (\aoneathree)(\bone)\Rightarrow(\delt).
\end{equation*}
Note that the left factor (\aoneathree) consists of two lines. 
Whenever the left factor consists of disjoint two lines, we must take an additional step before we highlight it.  First, move one of the lines from (\aoneathree) to (\bone) to decrease the partial ordering and obtain the word (\aone)(\afourathree). We have
\begin{equation*}
AB:=(\aone)(\afourathree) \Rightarrow (\aoneathree)(\bone).
\end{equation*}
Now we may highlight $A$:
\begin{equation*}
    (\aonehigh)(\afourathree).
\end{equation*}
Note that the edges of $B$ are contained in the highlighted edges of $A$. Move one of the edges from $B$ to $A$ to increase the partial ordering. We chose to move the leftmost edge and obtain:
\begin{equation*}
AB\Rightarrow A'B':=    (\aoneafour)(\athree).
\end{equation*}
Keep in mind that triangles are oriented clockwise, so we get $B'=$(\athree). Writing $B'=$(\bone) does not follow the orientation, so this is not allowed. Then, we rehighlight $A'$:
\begin{equation*}
    (\aoneafourhigh)(\athree).
\end{equation*}
Again, we see that the edges of $B'$ are contained in the highlighted edges of $A'$. Move the remaining edge from $B'$ to $A'$ to have more left weighted word $$A'B'\Rightarrow A''B'':=(\delt)(\e).$$ 
This concludes:
    (\aoneathree)(\bone)$\Rightarrow$ (\delt).
\end{example}

Pairs $(A,B)$ of canonical factors of 4-braids with $R(A)\cap S(B)\neq\emptyset$ (resp. $=\emptyset$) are listed in Table~\ref{table:1} (resp. Table~\ref{table:2}) below. 

\begin{example}  
In Table~\ref{table:1}, we have listed all possible pairs $(A_i, A_{i+1})$ of canonical factors up to rotation whose partial ordering can be increased. 
Namely, there are $A_i', A_{i+1}' \in \CF$ such that $A_i A_{i+1} \Rightarrow A_i' A_{i+1}'$. 
\end{example}
\begin{table}[h!]
\begin{center}
\begin{tabular}
{|p{1.3cm}||p{3.1cm}|p{3.5cm}|p{3.6cm}|}
 \hline
 \multicolumn{4}{|c|}{Order Increasable Multiplications  for $\CF$ up to rotation} \\

 \hline
 \diagbox[innerwidth=1.3cm,height=1.5cm]{$A_i$}{$A_{i+1}$}
& \mbox{Single Edge}
& \mbox{Disjoint Edges}
& \mbox{Triangle}
\\
 
\hhline{|=||=|=|=|}

 &(\aone)(\afour)$\Rightarrow$(\aoneafour)&(\bone)(\atwoafour)$\Rightarrow(\delt)^{**}$&(\aone)(\aoneafour)$\Rightarrow$(\aoneafour)(\btwo)\\
 
 &(\aone)(\bone)$\Rightarrow$(\atwoaone)&(\aone)(\aoneathree)$\Rightarrow$(\aoneathree)(\aone)&(\aone)(\atwoaone)$\Rightarrow$(\atwoaone)(\atwo)\\
 
 \centering\mbox{Single} &(\bone)(\atwo)$\Rightarrow$(\atwoaone)&(\aone)(\atwoafour)$\Rightarrow$(\aoneafour)(\atwo)&(\bone)(\atwoaone)$\Rightarrow$(\atwoaone)(\aone)\\

 \centering\mbox{Edge}&(\aone)(\athree)$\Rightarrow$(\aoneathree)&&(\aone)(\afourathree)$\Rightarrow$(\delt)\\
 
 &&&(\bone)(\aoneafour)$\Rightarrow$(\afourathree)(\btwo)\\

 &&&(\aone)(\athreeatwo)$\Rightarrow$(\aoneathree)(\atwo)\\
 
 \hline

 \centering\mbox{Disjoint}&(\aoneathree)(\bone)$\Rightarrow(\delt)^*$&&(\aoneathree)(\atwoaone)$\Rightarrow$(\delt)$(\atwo)^*$\\

\centering\mbox{Edges}&&&\\
 
 \hline
 
 &(\atwoaone)(\afour)$\Rightarrow$(\delt)&(\atwoaone)(\atwoafour)$\Rightarrow$(\delt)(\atwo)&(\atwoaone)(\afourathree)$\Rightarrow$(\delt)(\athree)\\

\mbox{Triangle}&&&(\atwoaone)(\aoneafour)$\Rightarrow$(\delt)(\btwo)\\

 \hline
\end{tabular}
\caption{Order increasable multiplications for $\CF$ up to rotation.\\
* see Example 2.16.\\
** see Example 2.15.}
\label{table:1}
\end{center}
\end{table} 

\begin{example}\label{ex:tau}
In addition to the list in Table~\ref{table:1}, if $A_i \in \CF$ and $A_{i+1}=(\delt)$, then $$A_i(\delt) \Rightarrow (\delt)\enspace\tau(A_i)$$  where $\tau:B_4\rightarrow B_4$ is the inner automorphism defined by 
$\tau(a_i)=\delta^{-1} a_i \delta = a_{i+1} $ and $\tau(b_i)=\delta^{-1} b_i \delta= b_{i+1}$. Thus, $\tau$ 
rotates a canonical factor diagram $90^\circ$ counterclockwise.
\end{example}

\begin{example}
In Table 2, we list all possible pairs $(A_i,A_{i+1})$ up to rotation whose partial order cannot be increased. Equivalently, $R(A_i)\cap S(A_{i+1})=\emptyset$. 
We omitted pairs of $(A_i,A_{i+1})$ where $A_i=A_{i+1}$. 
\begin{table}[h!]
\begin{center}
\begin{tabular}{|p{1.3cm}||p{2.5cm}|p{2.5cm}|p{2.5cm}|}
 \hline
 \multicolumn{4}{|c|}{Order Non-Increasing Multiplications for $B_4$ up to rotation} \\

 \hline
 \diagbox[innerwidth=1.3cm,width=1.6cm,height=1.3cm]{$A_i$}{$A_{i+1}$}
& \mbox{Single Edge}
& \mbox{Disjoint Edges}
& \mbox{Triangle}
\\

\hhline{|=||=|=|=|}

 &(\aone)(\atwo)&(\bone)(\aoneathree)&\\
 
 \centering\mbox{Single} &(\aone)(\btwo)&&\\
 
 \centering\mbox{Edge}&(\bone)(\aone)&&\\

 &(\bone)(\btwo)&&\\

 \hline

 \centering\mbox{Disjoint}&(\aoneathree)(\afour)&(\aoneathree)(\atwoafour)&(\aoneathree)(\aoneafour)\\

\centering\mbox{Edges}&(\aoneathree)(\btwo)&&\\

&(\aoneathree)(\aone)&&\\
 
 \hline
 
 &(\atwoaone)(\atwo)&(\atwoaone)(\aoneathree)&(\atwoaone)(\athreeatwo)\\

&(\atwoaone)(\aone)&&\\

 \centering\mbox{Triangle}&(\atwoaone)(\athree)&&\\

 &(\atwoaone)(\bone)&&\\

 &(\atwoaone)(\btwo)&&\\

 \hline

\end{tabular}
\caption{Maximally left weighted words $A_iA_{i+1}$ up to rotation. The cases where $A_i=A_{i+1}$ or $A_i=\delta$ are omitted.}
\label{table:2}
\end{center}
\end{table}

\end{example}

Finally, we extend the relation $\Rightarrow$ (Definition~\ref{def:Rightarrow}) to the set $\Wds(\beta)$. 

\begin{definition}
   For braid words $\delta^\ell A_1\cdots A_k\text{ and } \delta^mA_1'\cdots A_n'$
    in $\text{Wds}(\beta)$, we define
    \begin{equation*}
        \delta^\ell A_1\cdots  A_k\Rightarrow \delta^mA_1'\cdots A_n'
    \end{equation*}
    if one of the following is true: 
    \begin{enumerate}[a)]
        \item $\ell<m$
        \item $\ell=m$, $k=n+1$, and there exists some $i\in\{1,\dots,k\}$ such that $A_i=(\e)$, $A_j=A'_j$ for all $j<i$ and $A_j=A'_{j-1}$ for all $j>i$. 
        \item $\ell=m$, $k=n$, and there exists some $i\in\{1,\dots,k-1\}$ such that $A_j=A_j'$ for all $j\neq i,i+1$ and $A_iA_{i+1}\Rightarrow A_i'A_{i+1}'$.
    \end{enumerate}
With the relation, $(\Wds(\beta),\Rightarrow)$ becomes a partially ordered set.
\end{definition}

 \section{The Left Canonical Form}\label{sec:LCF-def} 

\subsection{Definition of the Left Canonical Form} 
In this section we introduce the left canonical form $\LCF(\beta)$ of a braid $\beta\in B_n$ as the maximal element in the partially ordered set $(\Wds(\beta), \Rightarrow)$. 

The word problem can be solved using the left canonical form by Xu for 3-braids \cite{Xu}, Kang, Ko and Lee for 4-braids \cite{KangKoLee}, and Birman, Ko, Lee for general n-braids \cite{Birman}. Namely, $\beta=\beta'$ if and only if $\LCF(\beta)=\LCF(\beta')$.

We first introduce $\inf(\beta), \sup(\beta), \ell(\beta)$ since 
the left canonical form $\LCF(\beta)$ will be defined to be the unique element of $\Wds(\beta)$ which 
\begin{enumerate}
    \item achieves the $\inf(\beta)$ and $\sup(\beta)$ simultaneously, and 
    \item is maximal in the partially ordered set $(\Wds(\beta), \Rightarrow)$.
\end{enumerate}

\begin{definition}\label{def:inf-sup}
For $\beta\in B_n$, define
\begin{align*}
    \inf (\beta)&:=\max\{r\in\mathbb{Z}\mid\delta^r\leq \beta\},\\
    \sup(\beta)&:=\min\{s\in\mathbb{Z}\mid \beta\leq\delta^s\},\\
    \ell(\beta)&:=\sup(\beta)-\inf(\beta).
\end{align*}
Here, $\ell(\beta)$ is called the \textit{canonical length} which is different from  $||\beta||$, the usual minimal word length in band generators. For example, $\ell(\atwoaone)=1$ and $||(\atwoaone)||=||a_2 a_1||=2.$
\end{definition}

Using these definitions, we can introduce the left canonical form.

\begin{theorem}\label{thm:LCF}
\cite{Xu, KangKoLee, Birman}
For any $n$-braid $\beta$, there exist unique $r\in\mathbb Z$, unique $k\in \mathbb Z_{\geq 0}$ and unique canonical factors $A_1,\dots,A_k\in\CFn \setminus \{e, \delta\}$ such that 
\begin{itemize}
\item
    $\beta=\delta^rA_1 A_2\cdots A_k$ as braid elements in $B_n$ and
\item
any consecutive pairs 
$A_1 A_2, \dots, A_{k-1}A_k$ are maximally left weighted. 
\end{itemize}
Moreover $r$ and $k$ satisfy 
$$\inf(\beta)=r, \quad \sup(\beta)=r+k, \ \mbox{ and } \ \ell(\beta)=k.$$
The unique factorization $\delta^rA_1 A_2\cdots A_k$ of $\beta$ is called the {\em left canonical form} of $\beta$ and denoted by $\LCF(\beta)$. 
\end{theorem}

\begin{remark} 
The left canonical form is {\em not} a conjugacy invariant. See Example \ref{ex_7_2}.
\end{remark}

We end the section by proving the ``if part'' of Theorem~\ref{thm:subset}. 

\begin{proposition}\label{prop:if-direction}
Let $A, B \in \CFn$. 
If there exists $Q \in \CFn$ such that $AQ=B$ then $\ch(A) \subset \ch(B)$. 
\end{proposition}

\begin{proof}
Assume that canonical factors $A, B, Q$ satisfy $AQ=B$. 
If the pair $(A, Q)$ satisfy all the conditions of $\diamond$ in Definition~\ref{def:diamond4}, so that $A Q = A\diamond Q$, then by the property (c) above Definition~\ref{def:diamond3}, it follows that $\ch(A) \subset \ch(B)$. 

If $(A, Q)$ does not satisfy the conditions of $\diamond$ in Definition~\ref{def:diamond4}, then $\LCF(AQ)$ has canonical length $\ell(AQ)=2$. However since $\LCF(AQ)=\LCF(B)=B$, the canonical length $\ell(AQ)=\ell(B)=1$, which is a contradiction.   
\end{proof}

\subsection{Left Canonical Form Algorithm}\label{sec:LCF}
In this section, we will review Kang, Ko, and Lee's algorithm for left canonical form \cite{KangKoLee} for 4-braids. It can be extended to general $n$-braids \cite{Birman}. 
This is a polynomial time algorithm which gives a unique left canonical form for a braid word $\beta$.

Let $\beta\in B_4$ be a braid element. Assume that $\beta$ is represented by a positive braid word in band generators;  
$\beta=A_1 A_2 \cdots A_k \in B_4^+$. Note that $A_i \in \BG(B_4) \subset \CF$ so $A_i \in \CF$. 
Change the composition of $\beta$ by the following rules.
\begin{enumerate}[\alph*)]
    \item If $A_i=(\e)$ for some $i$, then remove $A_i$ from the word $A_1\cdots A_k$. 
%
    \item Replace $A_i A_{i+1}$ with $A'_i A'_{i+1}$ so that $A_i A_{i+1} \Rightarrow A'_i A'_{i+1}$ (as listed in Table 1). 
\end{enumerate}
Apply (a) and (b) until the partial order cannot be raised anymore. The resulting word is the left canonical form, $\LCF(\beta)$, as defined in Theorem~\ref{thm:LCF}. We remark that the resulting word $\LCF(\beta)$ does not depend on the way we apply (a) and (b). One may start with increasing the partial order of any pair $(A_i,A_{i+1})$ in the braid word. 

If $\beta$ is not positive ($\beta\notin B_4^+$), we get rid of the negative exponent terms. 
Suppose $\beta=A_1^{\epsilon_1}A_2^{\epsilon^2}\cdots A_k^{\epsilon_k}$, where $A_i \in \CF$ and $\epsilon_i=\pm 1$. 
Suppose that $\epsilon_i=-1$ for some $i$. Then, replace $A_i^{-1}$ with $A_i'\delta^{-1}$ for some canonical factor $A_i' \in \CF$.
Up to rotation, we have:
\begin{align*}
    (\aone)^{-1}=(\afourathree)\delta^{-1}\\
    (\bone)^{-1}=(\atwoafour)\delta^{-1}
\end{align*}
Applying the rule $A_i'\delta^{-1}=\delta^{-1}\tau^{-1}(A_i')$ a number of times 
will shift $\delta^{-1}$ to the beginning of the word. 
Repeat this procedure until we exhaust all the factors with negative exponents. It will give us a word $\delta^{-r}P$ for some $r \in \mathbb N$ and positive word $P \in B_4^+$ in band generators. Finally we  apply the above algorithm to the positive word $P$ to obtain $\LCF(\beta$).

\begin{example}
    We will use the Diagrammatic Left Canonical Form Algorithm to find $\LCF(\beta)$, where
    \begin{align*}
        \beta&=b_2a_1b_1a_4a_2\\
        &=(\btwo)(\aone)(\bone)(\afour)(\atwo)\\
        &=A_1A_2A_3A_4A_5
    \end{align*}
    in band generators. We will look at each pair of $(A_i,A_{i+1})$ to determine if we can increase the partial order using part (b) of the algorithm.
    \begin{align*}
        \beta&=(\btwo)(\aone)(\bone)(\afourhigh)(\atwo)&\text{Highlight $A_4$}\\
        &=(\btwo)(\aone)(\bone)(\atwoafour)&\text{Increase partial order}\\
        &=(\btwo)(\aone)(\delt)&\text{Increase partial order (See Ex. 2.16)}\\
        &=(\delt)(\bone)(\atwo)&\text{Apply $ A_i(\delt)\Rightarrow(\delt)\tau(A_i)$}\\
        &=(\delt)(\bonehigh)(\atwo)&\text{Highlight $A_2'$}\\
        &=(\delt)(\atwoaone)&\text{Increase partial order}
    \end{align*}
    Thus, 
\begin{equation*}
    \LCF(\beta)=(\delt)(\atwoaone).
\end{equation*}
Note that it does not matter which pair $(A_i,A_{i+1})$ is considered first. 
\end{example}

\begin{example}\label{ex_7_2}
    Consider the knot $7_2$ represented by the 4-braid 
    \begin{equation*}
        \beta=a_1a_1a_1a_2a_1^{-1}a_2a_3a_2^{-1}a_3.
    \end{equation*}
    Using the LCF Algorithm, we find 
    \begin{equation*}
    \LCF(\beta)=\delta^{-1}a_4a_4a_4a_1b_2a_2a_3a_3.
    \end{equation*}
The knot $7_2$ is strongly quasi-positive, as it can be represented by another braid $\beta'=a_1a_1b_2b_1a_3$ which is strongly quasi-positive.
We can observe $\LCF(\beta')=\beta'$. For this example, $\beta$ and $\beta'$ are conjugate but $\LCF(\beta)\neq \LCF(\beta')$.
\end{example}

\subsection{Super Summit Set and The Left Canonical Form}
In this section, we review the definition for the super summit set, which is a conjugacy class invariant for a braid. We also give important properties of the super summit set in Proposition~\ref{prop:sss} and Theorem~\ref{thm:sss}. The conjugacy problem has been solved using band generators for $B_3$ by Xu \cite{Xu}, for $B_4$ by Kang, Ko, and Lee \cite{KangKoLee}, and for $B_n$ by Birman, Ko, and Lee \cite{Birman}.
\begin{definition}\label{def:inf[beta]}
For the conjugacy class $[\beta]$ of $\beta \in B_n$ we define: 
\begin{eqnarray*}
\inf[\beta]&:=& \max\{ \inf(\beta') \mid \beta' \mbox{ is conjugate to } \beta\} \\
\sup[\beta] &:=& \min\{ \sup(\beta') \mid \beta' \mbox{ is conjugate to } \beta\}
\end{eqnarray*}
\end{definition}

\begin{definition}\label{def:SSS}
The \textit{super summit set} of a braid $\beta$, denoted $\SSS(\beta)$, is the set of conjugates $\beta'$ of $\beta$ such that the canonical length
$\ell(\beta')$ is minimal among all the conjugates of $\beta$. 
\end{definition} 

In \cite[Corollary 4.6]{KangKoLee} it is proved that the $\inf[\beta]$ and $\sup[\beta]$ can be achieved simultaneously by the same element. Since the canonical length $\ell(\beta)=\sup(\beta)-\inf(\beta)$ we have the following: 

\begin{proposition}\label{prop:sss}
Any super summit element $\beta' \in \SSS(\beta)$ realizes $\inf[\beta]$ and $\sup[\beta]$. Namely $\inf[\beta]=\inf(\beta')$ and $\sup[\beta]=\sup(\beta')$. 
\end{proposition}

An important fact proved by Elrifai and Morton is that: 
two braids are conjugate if and only if their super summit sets are identical if and only if their super summit sets intersect \cite{ElrifaiMorton,Garside}. 

Let $W:=\LCF(\beta)=\delta^rA_1A_2\cdots  A_k$ be the left canonical form of $\beta\in B_n$. 
Let $\tau$ be the inner automorphism of $B_n$ defined by 
$$\tau(\beta)=\delta^{-1} \beta \delta.$$
In terms of the non-crossing partition diagram, $\tau$ rotates the canonical factor diagram $2\pi/n$ counterclockwise. 
Define the {\em cycling} $c(W)$ and the {\em decycling} $d(W)$ of the word $W$ as follows:
\begin{align*}
    c(W)&:=\delta^rA_2\cdots A_k\tau^{-r}(A_1)\\
    d(W)&:=\delta^r\tau^r(A_k)A_1\cdots A_{k-1}
\end{align*}

In general, $\beta\in B_n$ is not necessarily an element of $\SSS(\beta)$. 
The next theorem characterizes elements of $\SSS(\beta)$. 
\begin{theorem}[Kang, Ko, Lee \cite{KangKoLee} and Birman, Ko, Lee \cite{Birman}]\label{thm:sss}
Let $\beta\in B_n$ be an $n$-braid with canonical length $\ell(\beta)\geq 3$. 
Let $W=\LCF(\beta)$. If $\ell(W)=\ell(c(W))=\ell(d(W))$, then $\beta$ is an element of the super summit set {\rm SSS}$(\beta)$.
\end{theorem}

This is a key theorem to solving the conjugacy problem. For the solution to the conjugacy problem, readers may refer to \cite[Section 5]{Birman}.

\section{Detection of SQP and ASQP braids}\label{Sec:6}

In this section, we explore how the left canonical form looks like for (almost) strongly quasipositive braids. 

\begin{definition}\label{def:SQP}
A braid $\beta$ is a {\em strongly quasipositive braid} (SQP) if it can be represented by a word $W$ written as a product of positive powers of some of the band generators. A link $\mathcal{K}$ is a {\em strongly quasipositive link} if $\mathcal{K}$ can be represented by a strongly quasipositive braid. 
\end{definition}

\begin{definition}\label{def:ASQP}
A braid $\beta$ is an {\em almost strongly quasipositive braid} (ASQP) if it can be represented by a word $W$ written as a product of band generators and inverses such that the number of inverses used is at most one. 
A link $\mathcal{K}$ is an \textit{almost strongly quasipositive link} if $\mathcal{K}$ can be represented by an almost strongly quasipositive braid. 

Additionally, a braid $\beta$ is a \textit{strictly almost strongly quasipositive braid} if $\beta$ is almost strongly quasipositive but not strongly quasipositive. A link $\mathcal{K}$ is \textit{strictly almost strongly quasipositive} if $\mathcal{K}$ is almost strongly quasipositive but not strongly quasipositive.
\end{definition}

We have four results on SQP braids and ASQP braids. 

\begin{theorem}\label{Thm:SQ}
Let $n\geq 3$.
An $n$-braid $\beta\in B_n$ is strongly quasipositive if and only if $\inf(\beta)\geq 0$.
\end{theorem}
\begin{proof} 
($\Rightarrow$) Let $\beta\in B_n$ be strongly quasipositive. By definition, $\beta$ can be represented by a positive word, $W$, in band generators of $B_n$. Apply the algorithm for left canonical form (Section~\ref{sec:LCF}) to $W$ we obtain $\inf(\beta)\geq 0$. 

($\Leftarrow$) Suppose 
$\beta$ has $\inf(\beta)\geq 0$. Then, 
\begin{equation*}  \LCF(\beta)=\delta^rA_1A_2\cdots A_k
\end{equation*}
where $r\geq 0$, $A_i\in\CFn$, and $R(A_i)\cap S(A_{i+1})=\emptyset$. 
Since $\CFn \subset B_n^+$, each $A_i$ is a positive word in band generators. Therefore, $\beta \in B_n^+$, in particular $\beta$ is strongly quasipositive. 
\end{proof}

The corollary below follows from Proposition~\ref{prop:sss} on super summit set. 
\begin{corollary}\label{cor:sqpconj}
A braid $\beta\in B_n$ is conjugate to a strongly quasipositive braid if and only if every element $\beta' \in \SSS(\beta)$ has $\inf(\beta')\geq 0$.
\end{corollary}

In Example~\ref{ex_7_2}, we give a braid representative $\beta \in B_4$ of the knot that is not strongly quasipositive but is conjugate to a strongly quasipositive braid. The next theorem is the analog of Theorem \ref{Thm:SQ} in the almost strongly quasipositive case.

\begin{theorem}\label{Thm:ASQP}
Let $\beta\in B_n$ with $n\geq 3$. If $\beta$ is almost strongly quasipositive, then $\inf(\beta)\geq -1$. Further, if $\beta$ is strictly almost strongly quasipositive, then $\inf(\beta)= -1$.
\end{theorem}

\begin{proof}
Suppose $\beta\in B_n$ is almost strongly quasipositive. If $\beta$ is strongly quasipositive, we can apply Theorem \ref{Thm:SQ}. Therefore, we will assume that $\beta$ is strictly almost strongly quasipositive. Suppose that $\beta$ can be represented by a word, $W$, with exactly one negative band. In the algorithm for left canonical form, a negative band contributes $\delta^{-1}$. This implies that $\inf(W)\geq -1$. 
On the other hand, since $\beta$ is not strongly quasipositive, we have $\inf(\beta)<0$ by Theorem \ref{Thm:SQ}. Therefore, we obtain $\inf(W)=-1$. Since $\LCF(\beta)=\LCF(W)$, we get $\inf(\beta)=-1$.  
\end{proof}

Theorems \ref{Thm:SQ} and \ref{Thm:ASQP}, and Corollary \ref{cor:sqpconj} give the following: 

\begin{corollary}\label{cor:SSS-ASQP}
If $\beta\in B_n$ with $n\geq 3$ is conjugate to an ASQP (resp. strictly ASQP) braid, then  every element $\beta'\in\SSS(\beta)$ has $\inf(\beta')\geq -1$ (resp. $\inf(\beta')=-1$). 
\end{corollary}

For both Theorem~\ref{Thm:ASQP} and Corollary~\ref{cor:SSS-ASQP}, the converse direction does not hold in general. 
See Example~\ref{ex:strict-inequality} below. However, if we add an additional condition, the converse of Corollary~\ref{cor:SSS-ASQP} holds. See Theorem \ref{thm:strictlyasqp} below.

\section{The Bennequin Inequality and $\LCF(\beta)$}\label{Sec:7}

We define three {\em negative band numbers} $\nb(\beta), \nb[\beta], \nb(K)$ as follows: 

\begin{definition}\label{def-of-m}
For a braid $\beta\in B_n$ the minimal number of negative bands, $\nb(\beta)$, is defined as follows:
$$
\nb(\beta)=\min\left\{ k \, \middle\vert
    \begin{array}{l}
    \text{ $\beta$ is represented by a word in band}\\
    \text{generators containing $k$ negative bands }
    \end{array}
\right\}
$$
Similarly for the conjugacy class $[\beta]$ of $\beta$ we define:
$$
\nb[\beta]=\min\{\nb(\beta') \mid \text{ $\beta'$ is conjugate to $\beta$ } \}
$$
For a knot or link $K \subset S^3$ we define:
$$
\nb(K)=\min\left\{ \nb(\beta) \mid 
\text{ $\beta$ is a braid representative of $K$ }
\right\}
$$
\end{definition}
For a braid representative $\beta$ of a knot $K$, 
$$0\leq \nb(K) \leq \nb[\beta] \leq \nb(\beta).$$
We can characterize SQP braids and ASQP braids in terms of the negative band number. 
\begin{itemize}
\item
The braid $\beta$ is SQP if and only if $\nb(\beta)=0$. 
\item
$\beta$ is ASQP if and only if $\nb(\beta)\leq 1$.
\item
$\beta$ is strictly ASQP if and only if $\nb(\beta) =1$.
\end{itemize}

The invariant $\nb(K)$ is closely related to the Bennequin type inequalities. 
Bennequin \cite{Bennequin} showed the maximal self-linking number $\SL(K)$ of a given knot type $K$ is bounded above by $-\chi(K)=2 g(K)-1$ where $g(K)$ is the genus of the knot. 
$$\SL(K)\leq 2g(K)-1$$
Bennequin's inequality has been extended to concordance invariants including the slice genus $g_4(K)$, Ozsv\'ath-Szab\'o's tau-invariant $\tau(K)$ \cite{OsvathSzabo} and Rusmussen's $s$-invariant $s(K)$ \cite{Rasmussen} as follows. 
$$\SL(K) \leq 2 \tau(K) -1 \leq 2 g_4(K)-1 \leq 2 g(K)-1$$
$$\SL(K) \leq  s(K) -1 \leq 2 g_4(K)-1 \leq 2 g(K)-1$$
Ito and Kawamuro \cite{KeikoIto} defined the defect $\D(K)$ of the (original) Bennequin inequality 
$$\D(K)=\frac{1}{2}(2 g(K)-1-\SL(K)).$$
It is shown that $\D(K) \in \mathbb Z$ and $\D(K) \leq \nb(K)$. In fact, it is conjectured that $\D(K) = \nb(K)$. 

We will study the negative band numbers $\nb(\beta)$, $\nb[\beta]$, and $\nb(K)$ in terms of the left canonical form. 

\begin{lemma}\label{lem<0}
We have
$\nb(\beta)\geq 1$ if and only if $\inf(\beta)<0$. 
$($Equivalently, $\nb(\beta)=0$ if and only if $\inf(\beta)\geq 0.)$
\end{lemma}

\begin{proof}
Let $\beta\in B_n$ with $n\geq 3$ and suppose that $\nb(\beta)\geq 1$. Notice $\nb(\beta)\geq 1$ if and only if $\beta$ is not strongly quasipositive. Thus, $\nb(\beta)\geq 1$ is equivalent to $\inf(\beta)<0$ by Theorem~\ref{Thm:SQ}. 
\end{proof}

\begin{theorem}\label{thm:inequality}
Let $\beta\in B_n$ with $n\geq 3$ and $\nb(\beta)\geq 1$. 
$$
0< -\inf(\beta)= |\inf(\beta)| \leq \nb(\beta). 
$$
\end{theorem}

\begin{proof}
The first inequality follows from Lemma~\ref{lem<0}.

Suppose that $\beta$ is represented by a word $W$ in band generators with $\nb(\beta) (\geq 1)$ negative bands. Apply the algorithm for left canonical form (Section \ref{sec:LCF}) to $W$. The algorithm first replaces each negative band with an $A \delta^{-1}$ for some canonical factor $A \in \CFn$, then shifting $\delta^{-1}$ to the left. During the shifting process it is possible that some $\delta^{-1}$ vanishes if cancellation occurs. 
Therefore, we obtain $-\nb(\beta)\leq\inf(\beta)$.
\end{proof}

We study when the inequality $-\nb(\beta)\leq\inf(\beta)$ in Theorem~\ref{thm:inequality} becomes a strict inequality in Proposition \ref{prop:strictineq}.

\begin{example}\label{ex:strict-inequality}
We give an example of Theorem~\ref{thm:inequality} where strict inequality $-\nb(\beta)<\inf(\beta)$ holds.  
Consider a 4-braid $\beta:=a_3a_1^{-1}a_2^{-1}b_2b_1a_1b_2b_1a_3$. 
Running the LCF algorithm (Section~\ref{sec:LCF})
we obtain $$\LCF(\beta)= \delta^{-1}a_2a_3b_2b_1a_1b_2b_1a_3,$$ i.e., $\inf(\beta)=-1.$
Clearly, $\nb(\beta)\leq 2.$  
In the following, we will show $\nb(\beta)=2$.
\begin{figure}[h!]
    \centering
    \includegraphics[scale=0.7]{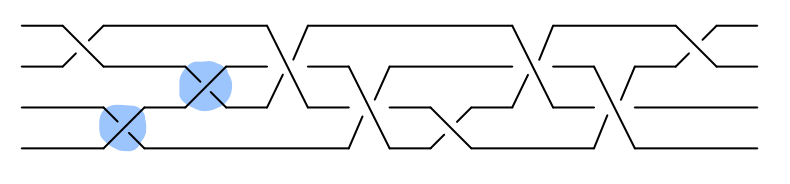}
    \caption{The braid $\beta:=a_3a_1^{-1}a_2^{-1}b_2b_1a_1b_2b_1a_3$.}
\end{figure}

To this end, we first show that the topological knot type of $\beta$ has braid index $4$. To do this, we will use the Morton-Franks-Williams Inequality \cite{Morton, FranksWilliams}. Let $K$ be the topological type of the braid closure $\hat{\beta}$. Let $i$ be the braid index of $K$. We will show $i(K)=4$, i.e., $\beta$ realizes the braid index of $K$. Let $E$ be the largest power and $e$ the smallest power of $a$ in the HOMFLY-PT polynomial of $K$. Then, the Morton-Franks-Williams inequality states
\begin{equation*}
    i\geq\frac{1}{2}(E-e)+1.
\end{equation*}
Mathematica's Knot Theory package computes the HOMFLY-PT polynomial for $K$ as
\begin{equation*}
    -a^{-6}+a^{-4}+a^{-2}+z^2a^{-8}-3z^2a^{-6}+4z^2a^{-4}-2z^4a^{-6}+3z^4a^{-4}-z^4a^{-2}+z^5a^{-4}.
\end{equation*}
In the above polynomial, we have $E=-2$ and $e=-8$. Therefore, the Morton-Franks-Williams inequality tells us that
\begin{equation*}
    i\geq 4.
\end{equation*}
Therefore, the braid index for $K$ is 4. Now we can proceed to our claim.\\

\noindent{\textbf{Claim:}} $\nb(\beta)\geq 2$. 
\begin{proof}
 First, we will show that $g(K)\leq 3$, where $g(K)$ is the genus of $K$. Consider a Bennequin surface, $\Sigma_W$, of $K$ coming from the braid $\beta$ represented by the word $W=a_3a_1^{-1}a_2^{-1}b_2b_1a_1b_2b_1a_3$ made of four disks joined together with nine twisted bands.
\begin{figure}[h!]
    \centering
    \includegraphics[scale=0.7]{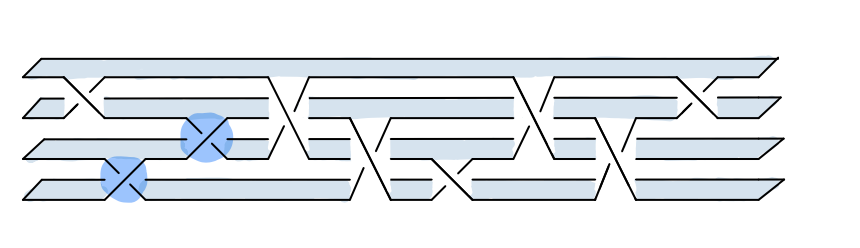}
    \caption{Bennequin Surface $\Sigma_W$}
\end{figure}

\noindent Therefore, 
\begin{equation*}
    \chi(K)\geq\chi(\Sigma_W)= 4-9=-5.
\end{equation*} 
Using the fact that $\chi(K)=1-2g(K)$ for any knot $K$, we find that $g(K)\leq 3$.

Next, we will show that $g(K)\geq 3$. Recall that
\begin{equation*}
    \text{degree}(\Delta_K(t))\leq 2g(K)
\end{equation*}
where $\Delta_K(t)$ is the Alexander polynomial of the knot \cite{Rolfsen}. Using Mathematica's Knot Theory Package, the Alexander polynomial of $K$ is given by
\begin{equation*}
    1-6t+17t^2-23t^3+17t^4-6t^5+t^6.
\end{equation*}
The degree of this polynomial is 6. Therefore, $g(K)\geq 3$.

We have shown that $g(K)=3$. Hence, $\chi(K)=-5$. Recall the Bennequin Inequality \cite{Bennequin}
\begin{equation*}
    \SL(K)\leq-\chi(K)
\end{equation*}
where $\SL(K)$ is the maximal self linking number of the link type $K$. By the truth of the Generalized Jones Conjecture \cite{DynnikovPrasolov,Keiko06,LaFountainMenasco}, $\SL(K)=-4+p-n$. Here, $p$ is the number of positive bands and $n$ is the number of negative bands in a braid representative which realizes the minimal braid index 4. In our example $\SL(K)=-4+7-2=1$. Observe the following inequality:
\begin{equation*}
    \SL(K)\leq-\chi(K)\leq -4+p+n
\end{equation*}
Subtracting $\SL(K)=-4+p-n$ from each side, we get
\begin{equation*}
    0\leq-\chi(K)- \SL(K)\leq2n
\end{equation*}
We have $\chi(K)=-5$ and $\SL(K)=1$. Therefore,
\begin{equation*}
    0\leq5-1\leq 2n
\end{equation*}
Therefore, $2\leq n$. This shows that $\beta$ must contain at least 2 negative bands, i.e., $\nb(\beta)\geq 2$.
\end{proof}
We may conclude $\nb(\beta)=2$. 

We note that $\nb(\beta)=2$ implies $\beta$ is not ASQP. Further, after computing cyclings and decyclings of $\LCF(\beta)$, we can apply Theorem 2.27 to verify that $\beta \in \SSS(\beta)$. We also have that $\inf(\beta)=-1$, which means the converse direction of Corollary~\ref{cor:SSS-ASQP} does not hold in general.
\end{example}

\section{The Negative Band Number and Reduction of $\LCF(\beta)$}\label{Sec:8}

The reduction operation for braids has been introduced by Kang, Ko and Lee \cite{KangKoLee} in order to solve the shortest word problem for 4-braids (see Theorem~\ref{thm:KKL-shortest}). 
We first review the reduction operation then apply it to investigate the relation between the negative band number $\nb(\beta)$ and the left canonical form $\LCF(\beta)$. 

\subsection{Review of Reduction Operation} 
\begin{definition}\label{def:Red}
Given a word  $W=\delta^rW_1W_2\cdots W_s$ with $\delta^{-1}<W_i<\delta$ and $W_i\neq e$, define $\red(W)$ as:
\begin{enumerate}[(1)]
    \item if $r\geq 0$ or $W_i<e$ for all $i=1,\dots,s$, then $\red(W)=W$;
    \item otherwise (i.e., $r<0$ and there exists some $i$ with $W_i\in\CF$), choose a word $W_k\in \CF$ whose word length $||W_k||$ is maximal among all $W_1, \dots, W_k$. 

By Corollary~\ref{cor:complementary} there exists $V_k \in \CF$ such that $W_k \diamond V_k = W_k V_k = \delta$. Put $W_k' = (V_k)^{-1}.$
We define
    \begin{equation*}
        \red(W)=\delta^{r+1}\tau(W_1)\cdots\tau(W_{k-1})W_k'W_{k+1}\cdots W_s,
    \end{equation*}
\end{enumerate}
Define $\Red(W):=\red^{|r|}(W)$, i.e., repeat the above algorithm until the exponent of $\delta$ is non-negative. 
\end{definition}

\begin{example}
   This example will show how to compute $\Red(W)$, where 
    \begin{align*}
        W&=\delta^{-2}(a_3a_2)(a_4a_3)a_4b_1b_2\\
        &=(\delt)^{-2}(\athreeatwo)(\afourathree)(\afour)(\bone)(\btwo)\\
        &=\delta^{-2}W_1W_2W_3W_4W_5
    \end{align*}
    Here, $r<0$, so we will proceed to option (2) in the algorithm. We will choose the word $W_2=(a_4a_3)$ of maximal word length among all $W_i$. Then,
    \begin{align*}
        W_2'&=\delta^{-1}W_2\\
        &=\delta^{-1}(a_4a_3)\\
        &=(a_4a_3a_2)^{-1}(a_4a_3)\\
        &=a_2^{-1}a_3^{-1}a_4^{-1}a_4a_3\\
        &=a_2^{-1}
    \end{align*}
    Thus,
    \begin{align*}
        \red(W)&=\delta^{-1}\tau(W_1)W_2'W_3W_4W_5\\
        &=\delta^{-1}(a_4a_3)a_2^{-1}a_4b_1b_2
    \end{align*}
    We will repeat this process one more time. Again, we will choose the word $(a_4a_3)$ of maximal length among all $W_i$. Then,
    \begin{equation*}
        \Red(W)=\red^2(W)=a_2^{-1}a_2^{-1}a_4b_1b_2.
    \end{equation*}
\end{example}

We note that both $\red(W)$ and $\Red(W)$ depend on choices. However, the next lemma shows that the word-length $||\Red(W))||$ is uniquely determined. 

\begin{lemma}
{\em (Kang, Ko, Lee \cite[Lemma 5.1]{KangKoLee})}
\label{lem:KKL-reduced}
Let $\beta\in B_n$ be an $n$-braid. For any word representative $W$ of $\beta$, every reduced word of the left canonical form minimizes the word length; namely, 
$$||\Red(\LCF(\beta))|| \leq ||W||.$$
\end{lemma}

Regarding the negative band number $\nb(\beta)$ and a shortest word we observe the following: 

\begin{lemma}\label{lem:nb(beta)}
Let $W$ be a word representing $\beta \in B_n$. 
The negative band number $\nb(\beta)$ is realized by $W$ if and only if $W$ gives a shortest word representing $\beta$.
\end{lemma}

\begin{proof}
Let $W$ be a shortest word representing $\beta$ and let $p$ (resp. $n$) be the number of positive (resp. negative) bands in the word $W$. 
Assume that a braid word $W'$ represents $\beta$ and realizes $\nb(\beta)$. 
Let $p'$ (resp. $n'$) be the number of positive (resp. negative) bands in the word $W'$.
We get $$n'=\nb(\beta)\leq n.$$

Since $W$ is a shortest word, 
$$
p'+n'=||W'|| \geq ||W||=p+n.
$$
Since $W$ and $W'$ belong to the same conjugacy class $[\beta]$, they have the same writhe. 
$$
p'-n'={\rm writhe}(W')={\rm writhe}(W)=p-n
$$
Therefore, $n' \geq n$. 

Above two paragraphs give $n=n'=\nb(\beta)$. 
\end{proof}

As a corollary of the above two lemmas, we observe the following: 

\begin{corollary}\label{cor:nb(beta)} 
Every reduced word $\Red(\LCF(\beta))$ achieves $\nb(\beta)$.
\end{corollary}

\subsection{Bounds of $\nb(\beta)$ in Terms of $\inf(\beta)$ and $\sup(\beta)$}

With the above preparation, we are now able to discuss the relation between $\nb(\beta)$ and $\inf(\beta)$. 
We first discuss when the inequality $-\nb(\beta)\leq\inf(\beta)$ in Theorem~ \ref{thm:inequality} becomes a strict inequality. 

\begin{proposition}\label{prop:strictineq}
Let $\beta\in B_n$ with $n\geq 3$ and $\nb(\beta)\geq 1$. 
If there is a shortest word representing $\beta$ that contains $A^{-1}$ for some canonical factor $A\in \CFn$ of word length $||A||\geq 2$ then we get a strict inequality $$-\nb(\beta)<\inf(\beta).$$ 
\end{proposition}

\begin{proof}
Let $W$ be a shortest word representing $\beta$ and containing $A^{-1}$ for some canonical factor $A\in \CFn$ of word length $||A||\geq 2$. 
By Lemma~\ref{lem:nb(beta)}, $W$ realizes the negative band number $\nb(\beta)$, and the contribution of $A^{-1}$ to the numerical value $\nb(\beta)$ is $||A||\geq 2$.

We note that $W$ is a non-positive word. 
According to the LCF algorithm (Section~\ref{sec:LCF}) for non-positive words, 
every canonical factor with negative exponent contributes one $\delta^{-1}$ to the left canonical form. 
In particular, the factor $A^{-1}\in W$ contributes $\delta^{-1}$ (or $\delta^0=e$ if some cancellation occurs) to $\LCF(\beta)$. 
In other words, $A^{-1}$ contributes 
$-1$ (or $0$ if some cancellation occurs) to $\inf(\beta)$. 

Comparing the contribution of $A^{-1}$ to $\nb(\beta)$ and $\inf(\beta)$, we obtain the desired strict inequality. 
\end{proof}

In the next theorem,
using the reduction operation we improve the inequality $-\nb(\beta)\leq\inf(\beta)$ in Theorem~\ref{thm:inequality} to an equation for 3-braids: 

\begin{theorem}\label{thm:3braid-nb}
Let $\beta$ be an n-braid.  
If $\inf(\beta)<0$ then 
$$\nb(\beta)\leq (n-2) |\inf(\beta)| - \min\{0, \sup(\beta)\}.$$ 
Moreover, the equality holds when $n=3$ and we have
$$\nb(\beta)= |\inf(\beta)| - \min\{0, \sup(\beta)\}.$$
\end{theorem}

\begin{proof}
Recall Definition~\ref{def:inf-sup} and 
put $r:=-\inf(\beta)=|\inf(\beta)|>0$ and $k:=\ell(\beta)\geq 0$. 
We have $$\sup(\beta)=-r+k.$$
Theorem~\ref{thm:LCF} implies that for some $A_1,\cdots,A_k \in \CFn \setminus \{e, \delta\}$
we have $\LCF(\beta)=\delta^{-r} A_1 \cdots A_k.$ 
By Corollary~\ref{cor:complementary} for each $i=1,\dots, k$ there exists a canonical factor $A_i' \in \CFn \setminus \{e, \delta\} $ such that $A_i A_i'=\delta$.

Recall the inner automorphism $\tau:B_n \to B_n$ defined by $\tau(\beta)=\delta^{-1}\beta\delta$. For a canonical factor $A\in \CFn$, $\tau(A)$ is diagrammatically counterclockwise $2\pi/n$ rotation of $A$ and $\tau(A^{-1})=(\tau(A))^{-1}$. 
Thus $\tau$ preserves the word length; $$||\tau(A)||=||A||=||\tau(A^{-1})||.$$

{\bf Case 1:} 
Suppose $-r+k\geq 0$.
Then $\min\{0, \sup(\beta')\}=0$. 
We apply the reduction operation to $\LCF(\beta)$ and obtain a $\Red(\LCF(\beta))$. By Definition~\ref{def:Red}, among the canonical factors $A_1, \cdots, A_k$ in $\LCF(\beta)$ the ones with $r$ largest word length, say $A_{i_1}, \cdots, A_{i_r}$, are replaced by the negative words ${(A'_{i_1})}^{-1}, \cdots, {(A'_{i_r})}^{-1}$ up to rotation by $\tau$.
The rest of the $k-r$ canonical factors in $\LCF(\beta)$ are kept the same up to rotation by $\tau$.

Since every canonical factor in $\CFn \setminus \{e, \delta\}$ has word length at most $n-2$ and $\tau$ preserves the word lengths of canonical factors (and their inverses), 
the number of negative bands in $\Red(W)$ is at most $(n-2)r$.
This gives $$\nb(\beta) \leq (n-2)r = -(n-2)\inf(\beta).$$ 

When $n=3$ we can say further.
Every canonical factor in $$\Cf\setminus\{e, \delta\}=\{a_1, a_2, a_3\}$$ has word length exactly $1$. This means each of ${(A'_{i_1})}^{-1},$ $\cdots,$ ${(A'_{i_r})}^{-1}$ corresponds to exactly one negative band and we obtain $$\nb(\beta) \leq r= -\inf(\beta).$$
On the other hand, Theorem~\ref{thm:inequality} gives $-\nb(\beta') \leq \inf(\beta')$. Therefore, for 3-braids we obtain $$-\nb(\beta)=\inf(\beta).$$

{\bf Case 2:}
If $-r+k <0$ we have $\min\{0, \sup(\beta)\}=\sup(\beta)$. 
Reduction operation gives: up to rotation by $\tau$ 
$$\Red(\LCF(\beta))=\delta^{-r+k} \ (A'_1)^{-1} \cdots (A'_k)^{-1}.$$ 
Recall Lemma \ref{lem:KKL-reduced} and Lemma~\ref{lem:nb(beta)}, which state that 
$\Red(\LCF(\beta))$ realizes the negative band number $\nb(\beta)$. Thus, we have
\begin{equation*}
    \nb(\beta) 
=||\delta|| (r-k)+ ||A_1'|| + \cdots + ||A_k'||.
\end{equation*}

Since the word length of $\delta$ is $||\delta||=n-1$ and $||A_i'||\leq n-2$, we get an upper bound of $\nb(\beta)$ 
\begin{eqnarray*}
\nb(\beta) 
&\leq& (n-1)(r-k)+(n-2)k \\
&=& (n-2)r + r -k \\
&=& -(n-2)\inf(\beta) - \sup(\beta).
\end{eqnarray*}
For 3-braids, since $||\delta||=2$ and $||A_i'||=1$, we get an explicit formula of $\nb(\beta)$ $$\nb(\beta)= 2(r-k)+k=r-(-r+k)=-\inf(\beta)-\sup(\beta).$$ 
\end{proof}

\section{Relations Between $\inf(\beta)$ and $\nb(\beta)$ for $\beta\in\SSS(\beta)$}\label{Sec:9}

The solution to the shortest word problem for 4-braids was given by Kang-Ko-Lee \cite{KangKoLee}. 
The result is a generalization of Xu's solution to the shortest word problem for 3-braids \cite{Xu}. 

\begin{theorem}
{\em (Kang, Ko, Lee \cite[Theorem 5.2]{KangKoLee})}
\label{thm:KKL-shortest}
Given a $4$-braid $\beta$, let $\beta'$ be an element in $\SSS(\beta)$. 
Then $\Red(\LCF(\beta'))$ gives a shortest word among all the conjugates of $\beta$. 
\end{theorem}

In this section, we explore the relation between $\inf(\beta)$ and $\nb(\beta)$ for a super summit element $\beta \in \SSS(\beta)$. 

Recall $\nb[\beta]=\min\{\nb(\beta') \mid \beta' \mbox{ is conjugate to } \beta\}$ in Definition~\ref{def-of-m}, the minimal number of negative bands for the conjugacy class of $\beta$. 
In the next lemma, we show that every shortest word realizes $\nb[\beta]$.

\begin{lemma}\label{lem:shortest}
Let $\beta\in B_n$. 
A word $W$ is a shortest word among all the conjugates of $\beta$ if and only if the number of negative bands in $W$ is $\nb[\beta]$.
\end{lemma}

\begin{proof}
Parallel argument in the proof of Lemma~\ref{lem:nb(beta)} applies. 
\end{proof}

The next lemma claims that super summit elements realize $\nb[\beta]$ for 3- and 4-braids. 

\begin{lemma}\label{lem:SSSand m}
Let $\beta$ be an $n$-braid where $n=3$ or $4$. 
For every super summit element $\beta' \in \SSS(\beta)$, we have
$\nb[\beta]=\nb(\beta')$. 
Moreover, 
a reduced word $\Red(\LCF(\beta'))$ contains $\nb(\beta')$ negative bands. 
\end{lemma}

\begin{proof}
Let $\beta'$ be a super summit element in $\SSS(\beta)$. 
By Theorem~\ref{thm:KKL-shortest} and Lemma~\ref{lem:shortest}, a reduced word $\Red(\LCF(\beta'))$ contains $\nb[\beta]$ negative bands. 

Since $\Red(\LCF(\beta'))=\beta'$ as braids, Corollary~\ref{cor:nb(beta)} implies that $\Red(\LCF(\beta'))$ contains $\nb(\beta')$ negative bands. 

Therefore, $\nb[\beta]=\nb(\beta')$.
\end{proof}

The next theorem shows that the converse direction holds for Corollary~ \ref{cor:SSS-ASQP} in the strictly almost strongly quasipositive case if we add an additional condition.

\begin{theorem}\label{thm:strictlyasqp}
A braid $\beta\in B_n$ with $n\leq 4$ $($see Remark~\ref{rem:n} below$)$
is conjugate to a strictly almost strongly quasipositive braid if and only if every element $\beta'\in \SSS(\beta)$ has $\inf(\beta')=-1$ and $\LCF(\beta')$ contains a canonical factor of word length $n-2$.
\end{theorem}

\begin{remark}\label{rem:n}
In the proof below, we note that the restriction on the braid index $n=3, 4$ is only required for the only-if ($\Rightarrow$) direction in which we use Lemma~\ref{lem:SSSand m} that is only proved for $n=3, 4$ at this writing.
The statement of the if-direction ($\Leftarrow$) holds for general $n$. 
\end{remark}

\begin{proof}
$(\Rightarrow)$ 
Suppose that $\beta$ is conjugate to a strictly almost strongly quasipositive braid. By Corollary~\ref{cor:SSS-ASQP} every super summit element $\beta'\in\SSS(\beta)$ has $\inf(\beta')=-1$. 
Denote $\LCF(\beta')=\delta^{-1} A_1 \cdots A_k$. 
Suppose that $A_i$ achieves the maximal word length among $A_1, \cdots, A_k$. 
Let $B_i \in \CF\setminus \{e, \delta\}$ be the canonical factor satisfying $A_i\diamond B_i=A_iB_i=\delta$ whose existence is guaranteed by Corollary~\ref{cor:complementary}. 
We apply the reduction operation (Definition~\ref{def:Red}) with respect to the $A_i$.  
We get 
\begin{eqnarray*}
\Red(\LCF(\beta'))&=& \delta^{-1+1} \tau(A_1)\cdots\tau(A_{i-1})(B_i)^{-1}A_{i+1}\cdots A_k\\
&=& \tau(A_1)\cdots\tau(A_{i-1})(B_i)^{-1}A_{i+1}\cdots A_k.
\end{eqnarray*}
By Lemma~\ref{lem:KKL-reduced}, 
$\Red(\LCF(\beta'))$ gives a shortest word representing $\beta'$.
By Lemma~\ref{lem:SSSand m}, the word $\Red(\LCF(\beta'))$ realizes the number $\nb(\beta')=\nb[\beta]$. 
Since $\beta$ is conjugate to a strictly ASQP braid, $\nb[\beta]=1$. 
The negative exponent term $(B_i)^{-1}$ in the reduced word $\Red(\LCF(\beta'))$ has word length $||(B_i)^{-1}||=||B_i||=1$ and contributes the unique negative band to the ASQP braid. 

The diagram $\B_i$ of the band generator $B_i$ is an edge connecting vertices of $D_n$. 
If $\B_i$ joins consecutive vertices then the diagram $\A_i$ has one connected component and $\A_i$ is an $n-1$ gon. 
Otherwise, $\A_i$ has two components, a $k$-gon and an $n-k$-gon for some $k$ joined by the edge $\B_i$.  
For both cases the word length of $A_i$ is $||A_i||=n-2$ (See Figure~\ref{fig:StrictlyASQP}).

\begin{figure}[h]
 \centering
\includegraphics[width=10cm]{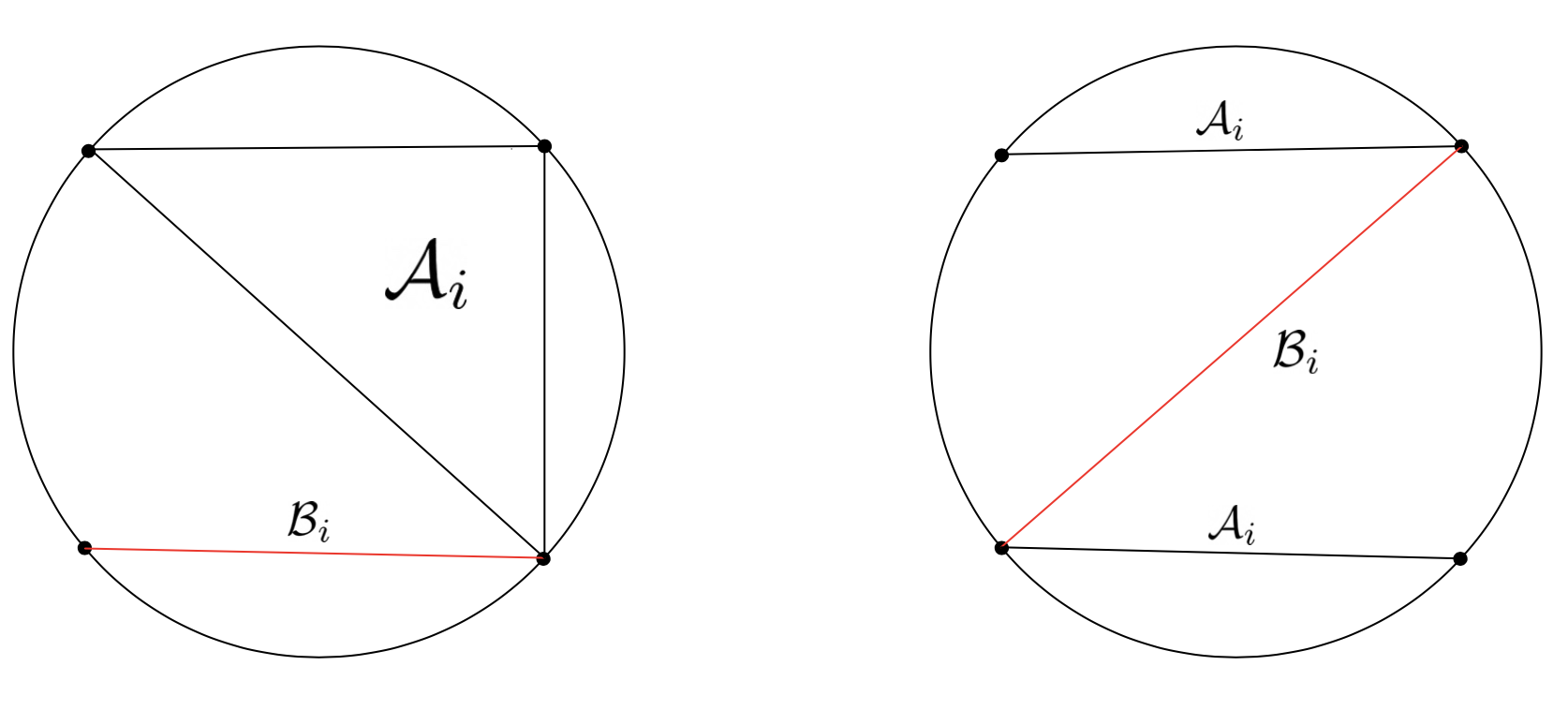}
 \caption{On the left figure, $A_i=a_3a_2$ has one connected component. On the right figure, $A_i=a_1a_3$ has two components. In both cases, the word length of $A_i$ is $||A_i||=n-2=2$.}
\label{fig:StrictlyASQP}
\end{figure}

$(\Leftarrow)$
Suppose that $\beta\in B_n$ 
and that every element $\beta'\in\SSS(\beta)$ has $\inf(\beta')=-1$ and $\LCF(\beta')=\delta^{-1}A_1 \cdots A_k$ contains some factor $A_i\in\CFn$ of word length $||A_i||=n-2$ for some $i$.  
By Corollary~\ref{cor:complementary} there is a canonical factor $B_i \in \CFn\setminus \{e, \delta\}$ such that $A_i \diamond B_i=A_iB_i=\delta$.

\begin{claim} We claim that $||B_i||=1.$    
\end{claim}

\begin{proof}[proof of the claim]
(Case 1) If the diagram $\A_i$ consists of a single connected component, then it is an $(n-1)$-gon. Since $A_i \diamond B_i = \delta$, the factor $B_i$ is a single band generator. Thus $||B_i||=1$.

(Case 2) 
If $\A_i$ has two connected components, say an $x$-gon $\X$ and a $y$-gon $\Y$ (i.e., $A_i=X \sqcup Y$), then $$n-2=||A_i||=||X||+||Y||=(x-1)+(y-1).$$ We get $x+y=n$. 
By Proposition~\ref{prop:property of ast} we know that $\X \ast \Y$ is an $(x+y)$-gon. On the other hand, $\delta$ is an $n$-gon. 
The only possibility is that the diagram $\B_i$ is the facing edge of $\X$ and $\Y$ so that  $\delta=A_i \diamond B_i= X \diamond (Y \diamond B_i)$. Thus $||B_i||=1$.   

(Case 3) 
Assume that the diagram $\A_i$ consists of three disjoint polygons $\X, \Y, \Z$. Suppose that $\X$ has $x$ sides, $\Y$ has $y$ sides and $\Z$ has $z$ sides. 
Since $||A_i||=n-2$, by the disjointness of $\X, \Y, \Z$, we have $n-2 = (x-1)+(y-1)+(z-1)$. 
We obtain $x+y+z=n+1.$ 
We may assume that $\X$ and $\Y$ are facing to each other (cf Definition~\ref{def:facing}) so that we can define $X \ast Y$. 

Apply Corollary~\ref{cor:A'} for the pair $A=X\sqcup Y\sqcup Z$ and $\delta$. 
Then we have $A \prec (X \ast Y) \ast Z \prec \delta$.  
By Proposition~\ref{prop:property of ast} we know the diagram $\X \ast \Y$ is an $(x+y)$-gon, and $(\X \ast \Y) \ast \Z$ is an $(x+y+z)$-gon. 
Since $x+y+z=n+1$ and $\delta$ is an $n$-gon, it is impossible that $\ch((\X \ast \Y) \ast \Z) \subset \ch(\delta)$. 
Therefore, $\A_i$ cannot have more than two components. 
\end{proof}

We continue the proof of the theorem. We have    
$$
\Red(\LCF(\beta'))=\tau(A_1)\cdots\tau(A_{i-1})(B_i)^{-1}A_{i+1}\cdots A_k.    
$$
Since $||B_i||=1$ this shows that $\nb[\beta]\leq 1$. 
If $\nb[\beta]=0$ then $\beta$ is conjugate to a SQP braid. 
By Corollary~\ref{cor:sqpconj} we obtain $\inf(\beta')=\inf[\beta] \geq 0$. This contradicts our assumption that $\inf(\beta')=-1$. Thus, $\nb[\beta]=1$. Therefore, $\beta$ is conjugate to a strictly almost strongly quasipositive braid.
\end{proof}

Recall the definition of $\inf[\beta]$ and $\sup[\beta]$ in Definition~\ref{def:inf[beta]} and that they are achieved by a super summit element simultaneously.  
The next two theorems follow from Proposition~\ref{prop:sss}, Theorem~\ref{thm:3braid-nb} and Lemma~\ref{lem:SSSand m}.

\begin{theorem}
Let $\beta$ be a 3-braid with $\inf[\beta]<0$. Then 
$$\nb[\beta]=|\inf[\beta]| - \min\{0, \sup[\beta]\}.$$ 
\end{theorem}

\begin{theorem}
Let $K$ be a knot or link in $S^3$ of braid index $n$. 
Let $\beta\in B_n$ be a braid representative of $K$. 
The following holds:
\begin{itemize}
\item
If $\inf[\beta]\geq 0$ then ${\mathcal D}(K) = \nb(K)=0$. 
\item
If $\inf[\beta]<0$ then
$$
|\inf[\beta]| \leq  \nb[\beta] \leq (n-2) |\inf[\beta]| - \min\{0, \sup[\beta]\}. 
$$
\end{itemize}
\end{theorem}
\section{Fractional Dehn Twist Coefficient}\label{sec:fdtc}
In this section we apply the dual Garside structure to compute the fractional Dehn twist coefficient (denoted FDTC) of a braid. 

The FDTC is a $\mathbb Q$-valued map $c: {\tt MCG}(S) \to \mathbb Q$ from the mapping class group of a surface $S$. 
See Honda, Kazez and Mati\'c's paper \cite{HondaKazezMatic} for the definition. 
Intuitively, this invariant quantifies how much twisting a diffeomorphism $\phi:S\to S$ possesses near a boundary component.

We will review some of the important properties of the FDTC.

\begin{proposition}\label{FDTC_properties} 
\cite{HondaKazezMatic}, \cite{ItoKawamuroEssential}
Let $C$ be a boundary component of $S$ and $\phi \in {\tt Aut}(S,\partial S)$. We have:
\begin{itemize}
    \item $c(\phi^n, C)= n \cdot c(\phi, C)$.
    \item $c(T_C, C)= 1$ and $c(\phi \circ T_C, C) = c(T_C \circ \phi , C)= 1+ c(\phi, C)$.
    \item $c(\phi, C)= c(\psi \circ \phi \circ \psi^{-1}, C)$ for any $\psi \in {\tt Aut}(S, \partial S)$.
\end{itemize}
\end{proposition}

To formalize the concept of twisting with respect to a boundary component, it is necessary to compare two arcs $\gamma$ and $\eta$ that share the same starting point $p\in C\subset \partial S$ with respect to the boundary component $C$. Informally, we want to define an order on arcs such that $\gamma \leq \delta$ if $\delta$ veers to the right of $\gamma$ when the chosen representatives realize the geometric intersection number. 

\begin{definition} \cite{HondaKazezMatic}
Let $\gamma, \eta \subset S$ be two distinct oriented arcs that start on the same point $p\in C\subset \partial S$. Isotope $\gamma$ and $\eta$ such that they minimally intersect transversely. Consider the tangent vectors of the arcs $\dot{\gamma}(0)$ and $\dot{\eta}(0)$. Then $\eta$ is to the right of $\gamma$, denoted $\gamma \leq \eta$, if the oriented basis $<\dot{\eta}(0), \dot{\gamma}(0)>$ agrees with the orientation on $S$. Equivalently, we pass to the universal cover $\tilde S$. Since we have isotoped $\gamma, \eta$ to minimally intersect in the universal cover, these will only intersect at $\tilde p$. Then $\eta$ is to the right of $\gamma$ if the interior of $\tilde{\eta}$ is in the region to the right of $\tilde{\gamma}$. 
\end{definition}

The following propositions are crucial when it comes to computing the FDTC. 

\begin{lemma}\label{Key}\cite{ItoKawamuroEssential}
Let $C\subset \partial S$ be a boundary component of $S$ and $\phi \in {\tt Aut}(S,\partial S)$. If there exists an essential arc $\gamma \subset S$ that starts on $C$ such that
$$ T^m (\gamma) \leq \phi^N(\gamma) \leq T^M(\gamma)$$
for some $m, N,  M\in \mathbb Z$ then $m/N \leq c(\phi, C)\leq M/N.$
\end{lemma}

\begin{theorem}\cite{ItoKawamuroEssential}, \cite{HondaKazezMatic}
The FDTC of $\phi\in {\tt Aut}(S,\partial S)$ defines a homogenous quasi-morphism of defect $1$. That is for all $\phi, \eta \in {\tt Aut}(S,\partial S)$
$$ |c(\phi \circ \eta, C ) - c(\phi, C) - c(\eta, C)|\leq 1 $$
and 
$$ c(\phi^n, C) = n\cdot c(\phi, C).$$

\end{theorem}

Since ${\tt MCG}(D_n) \cong B_n$ we can view a braid as a mapping class of the $n$-punctured disk $D_n$. 
Our goal is to find bounds on the fractional Dehn twist of a braid $\beta \in B_n$ using the $\LCF(\beta)$. Our surface is the $n$-punctured disk $D_n$ such that there is only one boundary component. As a result, we will denote the fractional Dehn twist of a braid $\beta$ by $c(\beta)$.  

\begin{lemma}\label{lemma:FDTC_CnF}
The FDTC satisfies the following: 
\begin{enumerate}
\item[$(1)$]
Let $\beta \in B_n$. If there exists an essential arc $\gamma$ such that $\beta(\gamma)$ and $\gamma$ are isotopic then $c(\beta)= 0$. 
\item[$(2)$]
$c(A)= 0$ for all $A\in \CFn \setminus \{\delta= \sigma_{n-1} \cdots \sigma_2 \sigma_1 \}$.
 \item[$(3)$]
$ c(\beta_1\beta_2)= c(\beta_2 \beta_1)$ for all $\beta_1, \beta_2 \in B_n.$
\item[$(4)$]
$c(\delta) = 1/n$.    
\item[$(5)$]
$c((\delt) (\aone))= 1/3$.
\end{enumerate}
\end{lemma}

\begin{proof}
(1) Notice that if $\phi (\alpha)\cong \alpha $ then $\alpha =  T^{0}_C(\alpha) \leq \phi(\alpha)\leq T^{0}_C(\alpha)= \alpha$. By Lemma~\ref{Key} we see that 
$$ 0 \leq c(\phi) \leq 0$$
giving us the result.

(2) This follows as a direct application of part (1); for all $A\in \CFn\setminus\{\delta \}$, there exists an essential arc from the boundary to the boundary that is preserved by the homeomorphism $\beta\in \CFn$. 
In fact, $A$ is reducible type in Thurston's classification. 

(3) By the conjugacy invariant property (Proposition~\ref{FDTC_properties}), we see that for any $\beta_1,\beta_2\in B_n$
$$
c(\beta_1 \beta_2)= c(\beta_1^{-1} (\beta_1 \beta_2) \beta_1) = c(\beta_2 \beta_1). 
$$

(4) Notice that $\delta^n = (\sigma_{n-1}\sigma_n ... \sigma_1)^n = \Delta^2 = T_C$. It follows that 
$$ c( \delta ) = \frac{1}{n} c(T_C)=\frac{1}{n}.$$

(5) Using the property that $A(\delt) = (\delt) \tau(A)$, where $\tau(A)$ is the rotation of $A\in \CF$ by $\pi/2$ as discussed in Example ~\ref{ex:tau}, we have
\begin{eqnarray*}
c(((\delt) (\aone))^3) &=& c((\delt)(\aone)(\delt)(\aone)(\delt)(\aone)) \\
&=& c((\delt)^3 (\athree) (\atwo) (\aone)) \\
&=& c((\delt)^4) = 1.
\end{eqnarray*}
Therefore, 
$c((\delt) (\aone)) = 1/3.$
\end{proof}

Using Lemma~\ref{lemma:FDTC_CnF} (alongside a train track idea for computing FDTC) we can compute the FDTC of a braid $\beta$. As an instructional example, we will compute the FDTC of a product of two canonical factors $A,B \in \CF$.

\begin{proposition}
$c((\bone)(\btwo))= 1/2$
\end{proposition}

\begin{proof}
See Figure~\ref{fig:traintrack}. Consider the arc $\alpha \subset D_4$ in the 4-punctured disk that goes from the boundary back to the boundary while encapsulating one of the punctures (the bottom left puncture). Similar to the concept of train tracks, we will do a zip move isotoping the endpoints of the arc to one point $p\in \partial D^2$ as in the right of Figure~\ref{fig:traintrack2}.

\begin{figure}
\includegraphics[width= 10cm]{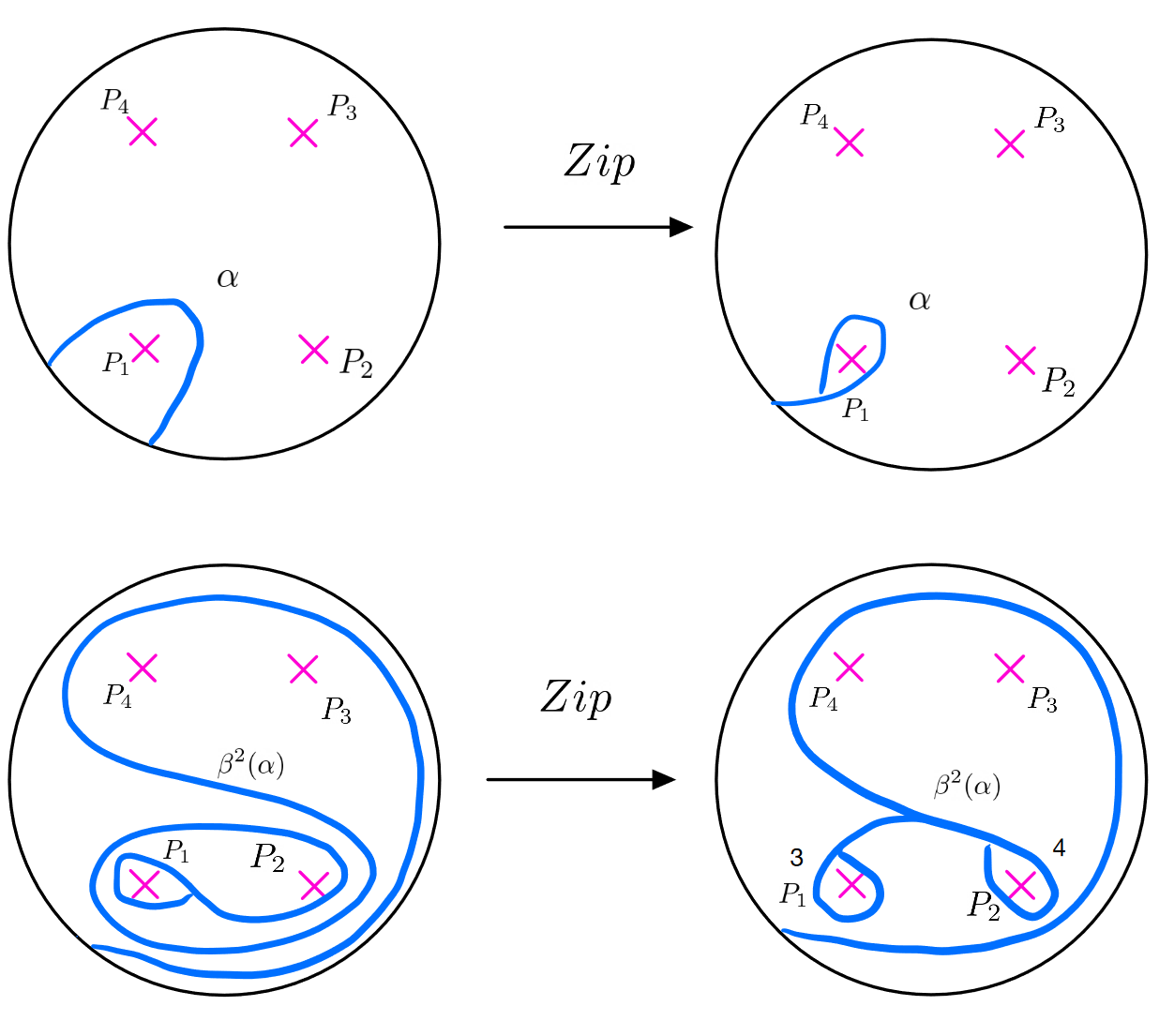}
\caption{This is the zip process the arcs $\alpha$ and $\beta^2(\alpha)$ into a train-track}
\label{fig:traintrack}
\end{figure}

After applying the braid $\beta = (\bone)(\btwo)$ twice, one sees that $\beta^2(\alpha)$ wraps around the bottom left puncture $P_1$ three times whereas it wraps around the bottom right puncture $P_2$ four times. Applying a zip move on the bottom face of the square made by the punctures, one gets a train-track of $\beta^2(\alpha)$ with labels $3$ and $4$ respectively. Further, note that $\beta^2(\alpha)\leq T_C(\alpha)$; it follows that $c(\beta)\leq 1/2$ by Lemma~\ref{Key}. 

After applying $(\bone)$, the resulting train track will be as in Figure~\ref{fig:traintrack2}. Notice that the resulting arc is to the right of $T_C(\alpha)$, the boundary Dehn twist on $\alpha$. 

\begin{figure}
\includegraphics[width=10cm]{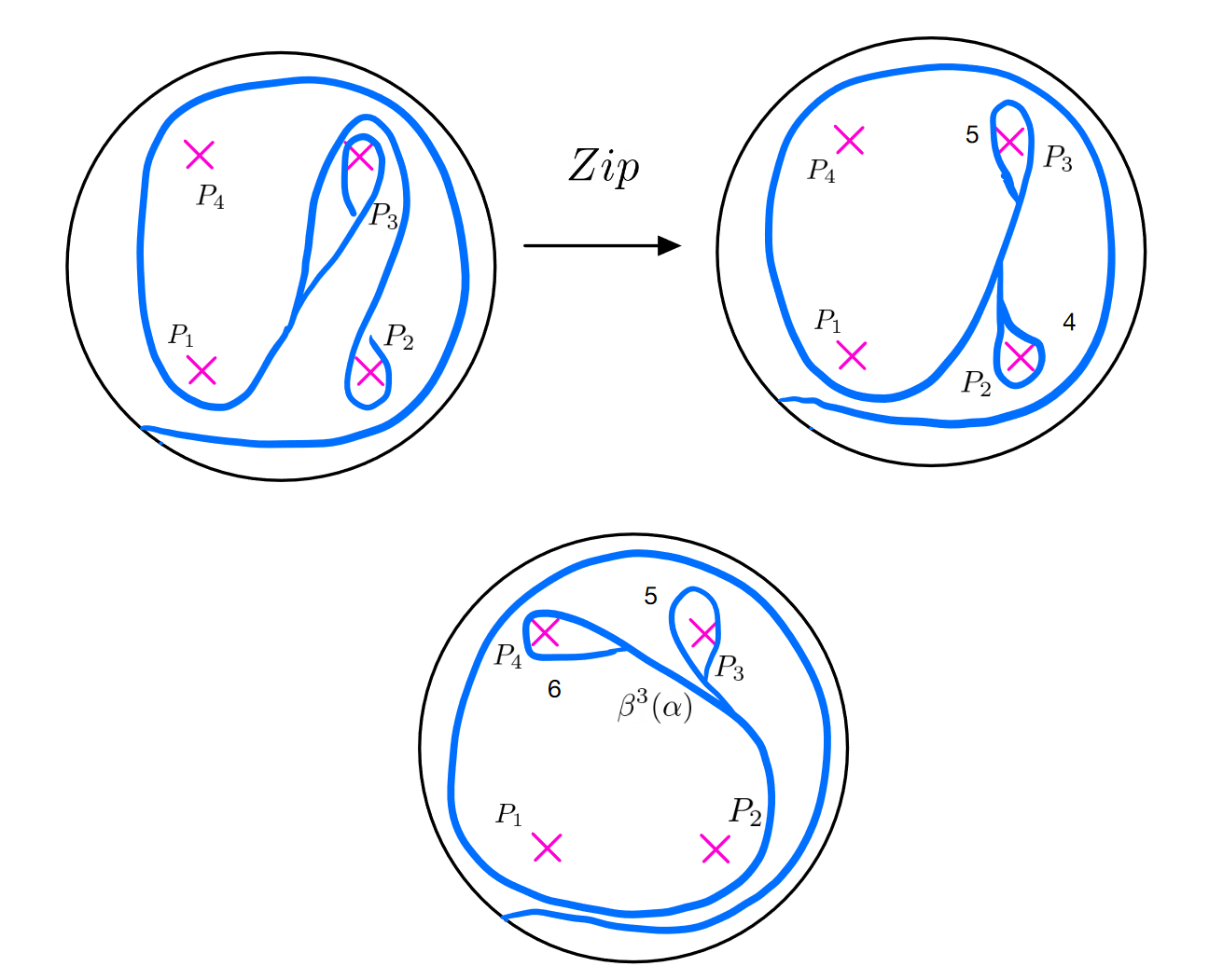}
\caption{On the top is the train track after applying the canonical factor $b_1$ to $\beta^2(\alpha)$; on the bottom is the train track for $\beta^3(\alpha)$.}
\label{fig:traintrack2}
\end{figure}

Similarly, applying $(\btwo)$ will increase the weights while passing over the second puncture after a first Dehn twist. It follows that $T_C(\alpha) \leq \beta^3(\alpha) $ and hence $1/3 \leq c(\beta)$. Using these train tracks, it is easy to see that for $n\in \mathbb N$
$$ T^{n+1}_C(\alpha)\leq \beta^{2n+ 3}(\alpha)$$
$$ \Rightarrow c((\bone)(\btwo))\geq \lim_{n\to \infty} \frac{n+1}{2n+3} = 1/2.$$
Combining the two inequalities, we get $c((\bone)(\btwo))= 1/2.$
\end{proof}

The following table is a list of the FDTC of the product $AB$ of canonical factors $A, B \in \CF$. 

\begin{table}[h!]
\begin{center}
\begin{tabular}{ | m{5em} | m{1cm}|| m{5em} | m{1cm} || m{5em} | m{1cm} | } 
  \hline
  Products & FDTC & Products & FDTC & Products & FDTC\\ 
  \hline
  (\atwoaone)(\atwo) & 0 & (\aone)(\atwo) & 0 & (\atwoaone)(\aone) & 0\\ 
  \hline
  (\aone)(\btwo) & 0 & (\bone)(\aone) & 0 & (\atwoaone)(\bone) & 0 \\ 
  \hline
  (\aoneathree)(\aone) & 0 & (\atwoaone)(\btwo) & 1/4 & (\aoneathree)(\btwo) & 1/4\\ 
  \hline
  (\bone)(\aoneathree) & 1/4 & (\aoneathree)(\afour) & 1/4& (\delt)(\aone) & 1/3\\ 
  \hline
  (\atwoaone)(\athree) & 1/3 & (\atwoaone)(\aoneathree) & 1/3 & (\aoneathree)(\aoneafour) & 1/3\\ 
  \hline
  (\delt)(\btwo) & 3/8 & (\atwoaone)(\athreeatwo) & 3/8 & (\bone)(\btwo) & 1/2\\ 
  \hline
  (\aoneathree)(\atwoafour) & 1/2 & (\delt)(\atwoafour) & 1/2 & (\delt)(\atwoaone) & 1/2\\ 
  \hline
\end{tabular}
\end{center}
\caption{The FDTC for $AB$ where $A, B \in \CF$.}
\end{table}

Suppose $\LCF(\beta) = \delta^r A_1 \cdots A_k$. Recall the invariants $\inf(\beta)=r$ and $\sup(\beta) = r+k$. 

\begin{proposition}\label{FDTC_bound}
For a braid $\beta$ we have 
$$\frac{\inf[\beta]}{n} \leq c(\beta) \leq \frac{\sup[\beta]}{n}.$$
\end{proposition}

\begin{proof}
Notice that $c(\delta)= 1/n$ by Lemma~\ref{lemma:FDTC_CnF}; it then follows by Lemma~\ref{Key} that
$$\frac{\inf(\beta)}{n} =\frac{r}{n} = c(\delta^r) 
\leq c(\delta^r A_1\cdots A_k) \leq c(\delta^{r+k})= \frac{r+k}{n}= \frac{\sup(\beta)}{n}.$$ 
Since $c(\beta), \inf[\beta], \text{ and }\sup[\beta]$ are conjugacy invariants (see Def~\ref{def:inf[beta]}) the statement follows. 
\end{proof}

It is proved by Malyutin \cite{Malyutin}  for braids that the Dehornoy floor  $\lfloor \beta \rfloor_{D}$ and the FDTC are numerically close to each other; 
$\lfloor \beta \rfloor_{D} \leq c(\beta) \leq \lfloor \beta \rfloor_{D}+1$. 
Since $\sup(\beta)\leq ||\beta||$, the minimal word length in band generators, 
our bound improves current bounds in the literature such as the following bound by Ito.

\begin{proposition}\cite{Itobound}.
If an $n$-braid $\beta$ is conjugate to a braid $\beta'$ represented by a product of $m$ band generators then
$$\lfloor \beta \rfloor_{D} < \frac{m+1}{n}.$$
\end{proposition}

\nocite{*}
\bibliographystyle{amsplain}
\bibliography{bib.bib}
\end{document}